\documentclass[12pt,abstracton]{article}
\usepackage{amsmath}
\usepackage{amsthm}
\newtheorem{counter1}{not to be used as environment} [section]
	\newtheorem{The}[counter1]{Theorem}
	
	\newtheorem{Lem}[counter1]{Lemma}
	\newtheorem{Cor}[counter1]{Corollary}
	
	\theoremstyle{definition}

	\newtheorem{Rem}[counter1]{Remark}

\usepackage{amsfonts}
\usepackage{amssymb}
\usepackage{dsfont} 
\usepackage{hyperref}
\usepackage{caption}
\usepackage{bbm}
\usepackage{nameref}
\usepackage{natbib}  
\usepackage{graphicx}
\usepackage{booktabs}
\usepackage{widetable}
\usepackage{tabularx}
\usepackage[toc,page]{appendix}
\usepackage[top=2cm,bottom=3cm]{geometry}

\begin{document}
\title{Sequential block bootstrap in a Hilbert space with application to change point analysis}
\author{Olimjon Sharipov\footnote{Institute of Mathematics, National University of Uzbekistan, 29 Dormon Yoli Str., Tashkent, 100125,Uzbekistan}, Johannes Tewes\footnote{Fakult\"{a}t f\"{u}r Mathematik, Ruhr-Universit\"{a}t Bochum, 44780 Bochum, Germany, Email address: Johannes.Tewes@rub.de} and Martin Wendler\footnote{Institut für Mathematik und Informatik, Ernst Moritz Arndt Universit\"{a}t Greifswald, Germany}}

\maketitle

\begin{abstract}
A new test for structural changes in functional data is investigated. It is based on Hilbert space theory and critical values are deduced from bootstrap iterations. Thus a new functional central limit theorem for the block bootstrap in a Hilbert space is required. The test can also be used to detect changes in the marginal distribution of random vectors, which is supplemented by a simulation study. Our methods are applied to hydrological data from Germany. 
\end{abstract}

\section{Introduction and main results}
\subsection{Introduction}

In the last decade statistical methods for functional data have received great attention, among them environmental data analysis, see \cite{HoeKo}. Due to a strong seasonal effect, for example in temperature or hydrological data, such time series are non-stationary and thus change point analysis is a complex topic. A possible solution is to look at annual curves instead of the whole time series. In this case, observations become functions. The method of functional principal components was used by \cite{KoMaSoZh} in testing for independence in the functional linear model and by \cite{BeHaeKn} in two sample tests for $L^2[0,1]$-valued random variables, a method that  was extended to change point analysis by \cite{BeGaHoKo}. Another approach is due to \cite{FrJuLiLl} who used record functions to detect trends in functional data.  In contrast to all former approaches, our method takes the fully functional observation into account. Whereas the statistic of \cite{BeHaeKn} is $\mathds R^d$-valued, our statistic depends directly on the functional or more generally Hilbert space-valued random variables. This becomes clear when considering the analogue of the CUSUM statistic, which takes the maximum of the norm of
\begin{align}
\sum_{i=1}^k X_i - \frac k n \sum_{i=1}^n X_i \ \ \text{for } k=1, \dots , n-1, \label{IntroCusum}
\end{align}
where $X_1, \dots , X_n$ are random variables taking values in a Hilbert space $H$. \\
Another change-point problem considers changes in the marginal distribution of random variables, now taking values in $\mathds R^d$. The advantage is that the type of the alternative (change in mean, change in scale,...) has not to be prespecified. The Kolmogorov Smirnov-type change point test was used for example by \cite{GoHo} and \cite{Ino} and is
\begin{align}
\max_{1 \leq m \leq n-1} \sup_{t \in [0,1]} \lvert \hat F_m(t) - \hat F_{m+1;n}(t) \rvert, \label{TestStatistikVerteilung}
\end{align}
where $\hat F_m$ and $\hat F_{m+1;n}$ are empirical distribution functions, based on $X_1,\dots ,X_m$ and $X_{m+1}, \dots X_n$, respectively. Define $Y_i$ by $Y_i(t) := 1_{ \{X_i \leq t \} }$ then (\ref{TestStatistikVerteilung}) equals
\begin{align*}
 \max_{1 \leq m \leq n-1} \lVert \bar{Y}_m -\bar{Y}_{m+1;n} \rVert_{\infty} ,
\end{align*}
where $\bar Y_m$ and $\bar Y_{m+1;n}$ are the sample means of $Y_1,\dots ,Y_m$ and $Y_{m+1}, \dots Y_n$, respectively. The $Y_i$ are no longer real valued random variables, but take values in a function space. Often one uses the space $D[0,1]$ of cadlag functions, however functional central limit theorems in $D[0,1]$ are difficult to obtain. Therefore in this paper we want to consider the Hilbert space $L^2$, equipped with the norm $\lVert \cdot \rVert = \sqrt{ \langle \cdot, \cdot \rangle}$, where $\langle \cdot, \cdot \rangle$ is the inner product of the Hilbert space. Using this norm instead of the supremums norm we get the statistic
\begin{align*}
 \max_{1 \leq m \leq n-1} \lVert \bar{Y}_m -\bar{Y}_{m+1;n} \rVert,
\end{align*}
which is a Cram\'{e}r-von Mises-type statistic. This $L^2$ approach to change-point analysis was also recently considered for independent observations by \cite{TsNi}. \\
Critical values for change-point tests are often deduced from asymptotics. The CUSUM statistic (\ref{IntroCusum}) can be expressed as a functional of the partial sum process
\begin{align*}
\sum_{i=1}^{\lfloor nt \rfloor} X_i \ \ \text{for } t \in [0,1], 
\end{align*} 
whose asymptotic behavior for $H$-valued data was investigated by \cite{ChWh} for mixingales and near epoch dependent processes. For statistical inference, one needs control over the asymptotic distribution. Due to dependence and the infinite dimension of the $\{X_i\}_{i \geq 1}$, the asymptotic distribution depends on an unknown infinite dimensional parameter - the covariance operator. Our solution is the bootstrap, which has been successfully applied to many statistics in the case of real or $\mathds R^d$-valued data. For Hilbert spaces, only \cite{PoRo} and recently \cite{DeShWe} established the asymptotic validity of the bootstrap. The results of \cite{PoRo} can only handle bounded random variables. Thus, indicator functions and statistics of type (\ref{TestStatistikVerteilung}) can be bootstrapped by their method, but general functional data cannot. \\
We extend the non overlapping block bootstrap by a sequential component, i.e. we are bootstrapping the partial sum process instead of the sample mean. This is inevitable for change-point problems, if the location of the possible change-point is unknown. \\
The paper is organized as follows: Sections \ref{Sec1} and \ref{Sec2}  contain the main results, an invariance principle for $H$-valued processes and the functional central limit theorem for bootstrapped data. Section \ref{Applications} describes the statistics and the bootstrap methodology for different change point tests including converging alternatives, while section \ref{SecRealLife} contains an analysis of two real life examples. In a simulation study, the finite sample behavior of the CUSUM test (for functional data) and the Cram\'{e}r-von Mises test (for $\mathds R$-valued data) is investigated and compared to the performance of existing tests. Proofs are provided in the appendix.

\subsection{Functional Central Limit Theorem for Hilbert space- valued functionals of mixing processes} \label{Sec1}

Let $H$ be a separable (i.e. there exists a dense and countable subset) Hilbert space with inner product $\langle \cdot , \cdot \rangle$ and norm $\lVert \cdot \rVert = \sqrt{\langle \cdot , \cdot \rangle}$. We say that an $H$-valued random variable $X$ has mean $\mu \in H$ if $E \langle X,h \rangle = \langle \mu , h\rangle$ for all $h \in H$. We denote it by $EX$. Moreover define the covariance operator $S \colon H \to H$ of $X$ (if it exists) by
\begin{align*}
\langle Sh_1 , h_2 \rangle = E \left [ \langle X- EX,h_1 \rangle \langle X - EX, h_2 \rangle \right ] \ \ h_1,h_2 \in H.
\end{align*}
For more details and a generalization to Banach spaces see the book of \cite{LeTa}. \\
Let $(\xi_i)_{i \in \mathds Z}$ be a stationary sequence of random variables, taking values in an arbitrary separable measurable space. A stationary sequence $(X_n)_{n \in \mathds Z}$ of $H$-valued random variables is called $L_p$-near epoch dependent ( NED($p$) ) on $(\xi_i)_{i \in \mathds Z}$, if there is a sequence $(a_k)_{k \in \mathds N}$ with $a_k \to 0$ as $k \to \infty$ and
\begin{align*}
E \left [ \lVert X_0 - E[ X_0 \vert \mathcal{F}_{-k}^{k} ] \rVert^p \right ] \leq a_k.
\end{align*}
Here $\mathcal{F}_{-l}^{m}= \sigma(\xi_{-l}, \dots , \xi_m)$ denotes the $\sigma$-field generated by $\xi_{-l}, \dots , \xi_m$. For the definition of conditional expectation in Hilbert spaces see \cite{LeTa}.\\
Concerning $(\xi_i)_{i \in \mathds Z}$, we will assume the following notion of mixing. Define the coefficients
\begin{align*}
\beta(k) = \left \lvert E \sup_{A \in \mathcal{F}_k^{\infty}} [ P(A \vert \mathcal{F}_{- \infty}^{0}) - P(A)] \right \rvert.
\end{align*}
$(\xi_i)_{i \in \mathds Z}$ is called absolutely regular if $\beta (k) \rightarrow 0$ as $k \to \infty$. 
\\\\
It is our aim to prove functional central limit theorems for $H$-valued random variables. Therefore, we will use the space $D_H[0,1]$, the set of all cadlag functions mapping from $[0,1]$ to $H$. An $H$-valued function on $[0,1]$ is said to be cadlag, if it is right-continuous and the left limit exists for all $x \in [0,1]$. Analogously to the real valued case we define the Skorohod metric
\begin{align*}
d(f,g) = \inf_{\lambda \in \Lambda} \left \{ \sup_{t \in [0,1]} \lVert f(t) - g \circ \lambda(t) \rVert + \lVert id - \lambda \rVert_{\infty} \right \} \ \ f,g \in D_H[0,1],
\end{align*}
where $\Lambda$ is the class of strictly increasing, continuous mappings of $[0,1]$ onto itself, $\lVert \cdot \rVert$ is the Hilbert space norm and $\lVert \cdot \rVert_{\infty}$ is the supremums norm. Moreover $id \colon [0,1] \to [0,1]$ is the identity function and $\circ$ denotes composition of functions.\\
Most topological properties on $D[0,1]=D_{\mathds R} [0,1]$ carry over to the space $D_H[0,1]$ (for more details on $D_{\mathds R} [0,1]$ see the book of \cite{Bil}). Equipped with the Skorohod metric $D_H[0,1]$ becomes a separable Banach space.
The limit process in our results will be Brownian motion. First define the Hilbert space analogue of a normal distribution. An $H$-valued random variable $N$ is said to be Gaussian, if for all $h \in H \setminus \{0\}$ the $\mathds R$-valued variable $\langle N, h \rangle$ has a normal distribution. The distribution of $N$ is uniquely determined by its mean and its covariance operator. A random element $W$ of $D_H[0,1]$ will be called Brownian motion in $H$ if
\begin{enumerate}
\item [(i)] $W(0) = 0$ almost surely,
\item [(ii)] $W \in C_H[0,1]$ almost surely, where $C_H[0,1]$ is the set of all continuous functions from $[0,1]$ to $H$,
\item [(iii)] the increments on disjoint intervals are independent, 
\item [(iv)] for all $0 \leq t < t+s \leq 1$ the increment $W(t+s) - W(t)$ is Gaussian with mean zero and covariance operator $sS$. $S \colon H \to H$ does not depend on $s$ or $t$.
\end{enumerate}
Note that the distribution of a Brownian motion $W$ is uniquely determined by the covariance operator $S$ of $W(1)$. 
\\\\
The first result states convergence of the partial sum process. Such a result was given by 
\cite{Wal} for martingale difference sequences and by \cite{ChWh} in the near epoch dependent case. They assume strong mixing, which is more general than absolute regularity. Then again, we require $L_1$-near epoch dependence, while they use $L_2$-near epoch dependence, which implies our conditions and is therefore more restrictive.

\begin{The} \label{FCLTHilbertraum}
Let $(X_n)_{n \in \mathds{Z}}$ be $L_1$-near epoch dependent on a stationary, absolutely regular sequence $(\xi_n)_{n \in \mathds{Z}}$ with $EX_1 = \mu \in H$ and assume that the following conditions hold for some $\delta > 0$ 
\begin{enumerate}
\item $E \lVert X_1 \rVert^{4 + \delta} < \infty$,
\item $ \sum_{m=1}^{\infty} m^2 (a_m)^{\delta / (\delta+3)} < \infty$,
\item $ \sum_{m=1}^{\infty} m^2 (\beta(m))^{\delta / (\delta+4)} < \infty$.
\end{enumerate}
Then
\begin{align*}
\left( \frac{1}{\sqrt{n}} \sum_{i=1}^{\lfloor nt \rfloor} (X_i - \mu ) \right)_{t \in [0,1]} \Rightarrow \left(  W(t) \right)_{t \in [0,1]}, 
\end{align*}
where $(W(t))_{t \in [0,1]}$ is a Brownian motion in $H$ and $W(1)$ has the covariance operator $S \colon H \to H$, defined by
\begin{align}
\langle S x, y \rangle = \sum_{i=-\infty}^{\infty} E [\langle X_0 - \mu,x \rangle \langle X_i - \mu, y \rangle ], \ \ \text{for } x,y \in H. \label{KovarianzOperator}
\end{align}
Furthermore, the series in (\ref{KovarianzOperator}) converges absolutely.
\end{The}

\subsection{Sequential Bootstrap for $H$-valued random variables.} \label{Sec2}

Theorem 1 has some applications, for example change-point test (see section \ref{Applications}). However, the problem arises, that the limiting distribution may be unknown, or even if it is known, it depends on an infinite dimensional parameter, in our case the covariance operator $S$. \\
To circumvent this problem, we will use the non overlapping block bootstrap of \cite{Car} to construct a process with the same limiting distribution as $\frac{1}{\sqrt{n}} \sum_{i=1}^{\lfloor nt \rfloor} (X_i - \mu )$. \\
For a block length $p(n)$, consider the $k= \lfloor n/p \rfloor$ blocks $I_1, \dots , I_k$, defined by
\begin{align*}
I_j = (X_{(j-1)p +1}, \dots , X_{jp}) \ \ \ \ j=1,2, \dots, k.
\end{align*}
Then we draw $k$ times independently and with replacement from these blocks. The drawn blocks (bootstrap blocks) build up a bootstrap sample and satisfy 
\begin{align*}
 P \left ( (X_{(j-1)p +1}^{\ast}, \dots , X_{jp}^{\ast}) = I_i \right ) = \frac{1}{k} \ \ \text{for} \ \ i,j = 1, \dots , k.
 \end{align*}
Now we can define a bootstrapped version of the partial sum process
\begin{align}
W_{n,p}^{\ast} (t) = \frac{1}{\sqrt{kp}} \sum_{i=1}^{\lfloor kpt \rfloor} (X_i^{\ast} - E ^{\ast} X_i^{\ast}). \label{BootstrappedPartialsummenprozess}
\end{align}
As usual, $E^{\ast}$ and $P^{\ast}$ denote conditional expectation and probability, respectively, given $\sigma (X_1, \dots, X_n)$. Further, $\Rightarrow_{\ast}$ denotes weak convergence with respect to $P^{\ast}$.
 The next result establishes the asymptotic distribution of the process $W_{n,p}^{\ast} (t)$, defined in (\ref{BootstrappedPartialsummenprozess}).

\begin{The} \label{SeqBootstrapHilbertraum}
Let $(X_n)_{n \in \mathds{Z}}$ be $L_1$-near epoch dependent on a stationary, absolutely regular sequence $(\xi_n)_{n \in \mathds{Z}}$ with $EX_1 = \mu$ and assume that the following conditions hold for some $\delta > 0$
\begin{enumerate}
\item $E \lVert X_1 \rVert^{4 + \delta} < \infty$,
\item $ \sum_{m=1}^{\infty} m^2(a_m)^{\delta / (\delta+3)} < \infty $,
\item $ \sum_{m=1}^{\infty} m^2 (\beta (m))^{\delta / (\delta+4)} < \infty $.
\end{enumerate}
Further, let the block length be nondecreasing, $p(n) = O(n^{1-\epsilon})$ for some $\epsilon$ and $p(n) = p(2^l)$ for $n= 2^{l-1} +1, \cdots , 2^l$, for all $l \in \mathds N$. Then
\begin{align*}
\left( W_{n,p}^{\ast}(t) \right)_{t \in [0,1]} \Rightarrow_{\ast} \left(  W(t) \right)_{t \in [0,1]} \ \ \text{a.s.} \ \ ,
\end{align*}
where $(W(t))_{t \in [0,1]}$ is a Brownian motion in $H$ and $W(1)$ has the covariance operator $S \colon H \to H$, defined in Theorem \ref{FCLTHilbertraum}.
\end{The}

\section{Application to change point tests} \label{Applications}

\subsection{Change in the mean of $H$-valued data} \label{ChangeInMeanSection}

Let us consider the following change point problem. Given $X_1, \dots, X_n$, we want to test the null hypothesis
\begin{align*}
\text{H}_0 \colon \ \ \ \ EX_1 = \dots = EX_n
\end{align*}
against the alternative 
\begin{align*}
\text{H}_A \colon \ \ \ \ EX_1 = \dots = E X_{k} \not = E X_{k+1} = \dots = EX_n,
\end{align*}
for some $k \in \{1, \dots , n-1\}$. \\
For real-valued variables, asymptotics of CUSUM-type tests have been extensively studied by \cite{CsHo}. They investigated tests for i.i.d. data, weakly dependent data and for long range dependent processes. The third case was extended by \cite{DeRoTa}.\\
For functional data, \cite{BeGaHoKo} have developed estimators and tests for a change point in the mean, which is extended by \cite{HoeKo} and \cite{AsKi} to weakly dependent data. They use functional principal components, while - motivated by Theorems 1 and 2 - we bootstrap the complete functional data. Consider the test statistic
\begin{align*}
T_n = \max_{ 1 \leq m \leq n-1} \frac{1}{\sqrt n} \left \lVert \sum_{i=1}^m X_i - \frac{m}{n} \sum_{i=1}^n X_i \right \rVert
\end{align*}
and its bootstrap analogue
\begin{align*}
T_n^{\ast} = \max_{ 1 \leq m \leq kp-1} \frac{1}{\sqrt {kp}} \left \lVert \sum_{i=1}^m X_i^{\ast} - \frac{m}{kp} \sum_{i=1}^{kp} X_i^{\ast} \right \rVert.
\end{align*}
The next result states that $T_n$ and $T_n^{\ast}$ have the same limiting distribution, which is a direct consequence of Theorems\ref{FCLTHilbertraum} and \ref{SeqBootstrapHilbertraum} and the continuity of both the maximum function and the Hilbert space norm.

\begin{Cor} \label{LimitTheoremCusumHilbertraum}
(i) Under the conditions of Theorem \ref{FCLTHilbertraum} 
\begin{align*}
T_n \Rightarrow \max_{t \in [0,1]} \lVert W(t) - t W(1)\rVert,
\end{align*} 
where $(W(t))_{t \in [0,1]}$ is the Brownian motion defined in Theorem \ref{FCLTHilbertraum}. \\
(ii) Under the conditions of Theorem \ref{SeqBootstrapHilbertraum}
\begin{align*}
T_n^{\ast} \Rightarrow_{\ast} \max_{t \in [0,1]} \lVert W(t) - t W(1)\rVert \ \ \ \ \text{a.s.}
\end{align*} 
\end{Cor}

Next we derive the asymptotic distribution of the (bootstrapped) change-point statistic under a sequence of converging alternatives. Define the triangular array of $H$-valued random variables
\begin{align*}
Y_{n,i} = \begin{cases}
     X_i & \text{if } i \leq \lfloor n \tau \rfloor \\
     X_i + \Delta_n & \text{if } i >  \lfloor n \tau \rfloor
     \end{cases}
\end{align*}
for $n \in \mathds N$ and $i \leq n$. Here, $\lfloor n \tau \rfloor$  is the unknown change-point for some $\tau \in (0,1)$ and $(\Delta_n)_n$ is an $H$-valued  deterministic sequence with
\begin{align*}
\lVert \sqrt{n} \Delta_n - \Delta \rVert \rightarrow 0,
\end{align*}
for $ n \to \infty$ and some $\Delta \in H$.\\
Now we want to test the Hypothesis $\Delta_n =0$ against the sequence of Alternatives where $\Delta, \Delta_n \in H \setminus \{0\}$. \\
Note that a bootstrap sample $(Y_{n,i}^{\ast})_{i \leq kp, n \geq 1}$ can be created analogously  to $(X_{i}^{\ast})_{i \leq kp}$. Then we can define the statistics $T_n$ and $T_n^{\ast}$, now based on $Y_{n,i}$ and $Y_{n,i}^{\ast}$, respectively.

\begin{Cor} \label{LimitTheoremCusumHilbertraumAlternative}
(i) Consider an array $(Y_{n,i})_{n \in \mathds N, i \leq n}$. If the conditions of Theorem \ref{FCLTHilbertraum}  hold for $(X_i)_{i \geq 1}$, then under the sequence of local alternatives
\begin{align}
T_n \Rightarrow \max_{t \in [0,1]} \lVert W(t) - t W(1)+ \phi_{\tau}(t) \Delta \rVert, \label{CUSUMNormalRot}
\end{align} 
where $(W(t))_{t \in [0,1]}$ is the Brownian Motion defined in Theorem \ref{FCLTHilbertraum}  and the function $ \phi_{\tau} \colon [0,1] \to \mathds R$ is defined by
\begin{align*}
\phi_{\tau} (t) = \begin{cases} 
 t (1-\tau) & \text{if } t \leq \tau \\
 (1-t)\tau & \text{if } t > \tau.
 \end{cases}
 \end{align*}
(ii) If the conditions of Theorem \ref{SeqBootstrapHilbertraum} are satisfied, then under the sequence of local alternatives
\begin{align}
T_n^{\ast} \Rightarrow_{\ast} \max_{t \in [0,1]} \lVert W(t) - t W(1)\rVert \ \ \text{a.s.} \ \ . \label{CUSUMBootsRot}
\end{align} 
\end{Cor}

The above Corollaries motivate the following test procedure, which is typical for bootstrap tests:
\begin{enumerate}
\item[(i)] Compute $T_n$.
\item[(ii)] Simulate $T_{j,n}^{\ast}$ for $j=1, \dots , J$.
\item[(iii)] Based on the independent (conditional on $X_1, \dots , X_n$) random variables $T_{n,1}^{\ast}, \dots , T_{n,J}^{\ast}$, compute the empirical $(1-\alpha)$- quantile $q_{n,J}(\alpha)$.
\item[(iv)] If $T_n > q_{n,J}(\alpha)$ reject the null hypothesis.
\end{enumerate}
By Corollary \ref{LimitTheoremCusumHilbertraum} and the Glivenko-Cantelli Theorem, the proposed test has an asymptotically significance level of $\alpha$, whereas by Corollary \ref{LimitTheoremCusumHilbertraumAlternative}, it has asymptotically nontrivial power. The deterministic element $\Delta \in H$ describes the amount of the change, while $\phi_{\tau}$ describes its location. Together they discriminate the limits of (\ref{CUSUMNormalRot}) and (\ref{CUSUMBootsRot}) and hence they are responsible for the asymptotic power. Note that the maximum of $\phi_{\tau}$ is $\tau (1 -\tau)$. Thus the power decreases drastically if the change occurs near the beginning of the observation period. \\
The above test problem is that of at most one change point (AMOC). However, especially in functional time series multiple changes are thinkable. Our statistic can be extended to allow such alternatives in the same way as the classical CUSUM statistic, see \cite{ErLo}.

\subsection{Change in the marginal distribution}

We will now apply the results to random variables, whose realizations are not truly functional. Consider, for example, the real valued random variables $X_1, \dots , X_n$ and the problem testing for changes in the underlying distribution:
\begin{align*}
\text{H}_0 \colon \ \ \ \ P(X_1 \leq t) = \dots = P(X_n \leq t ) \ \ \forall t \in \mathds R
\end{align*}
against  
\begin{align*}
\text{H}_A \colon \ \ \ \ P(X_1 \leq t) = \dots = P( X_{k}\leq t) \not = P( X_{k+1} \leq t) = \dots = P(X_n \leq t),
\end{align*}
for some $k \in \{1, \dots , n-1\}$ and $t \in \mathds R$. \\
Asymptotic tests have been investigated by \cite{CsHo}, \cite{HoSh} and \cite{Szy} in the independent case, by \cite{Ino} for strong mixing data, and by \cite{GiLeSu} for long-memory linear processes. Common test statistics depend on the empirical distribution function, and therefore on the indicators
\begin{align}
1_{ \{ X_i \leq t \} }, \ \ t \in \mathds R. \label{Indikatorfunktion}
\end{align}
 Those can be interpreted as random functions and hence random elements of the Hilbert space of functions $f \colon \mathds R \to \mathds R$, equipped with the inner product
\begin{align*}
\langle f, g \rangle_w = \int_{\mathds R} f(t)g(t) w(t) \ dt
\end{align*}
for some positive, bounded weight function with $\int_{\mathds R} w(t) \ d t < \infty$. \\
By Fubini's Theorem, we have
\begin{align*}
E \left [ \int_{\mathds R} 1_{ \{ X_i \leq t \} } h(t) w(t) \ dt \right ] = \int_{\mathds R}F(t) h(t) w(t) \ dt \ \ \forall h \in H.
\end{align*}
Hence by the definition, it follows that the mean of  (\ref{Indikatorfunktion}) is just the distribution function of $X$. So the change in the mean-problem (in $H$) becomes a change in distribution-problem (in $\mathds R$). \\
Furthermore, the arithmetic mean becomes the empirical distribution function. Note that this still holds when we consider $\mathds R^d$-valued data, which leads to the following test statistic
\begin{align}
T_{n,w} = \max_{ 1 \leq m \leq n-1} \frac{1}{n} \int_{\mathds R^d} \left( \sum_{i=1}^m 1_{ \{X_i \leq t \} } - \frac{m}{n} \sum_{i=1}^n 1_{ \{X_i \leq t \} } \right)^2 w(t) \ dt. \label{UnsereTestStatistikVerteilung}
\end{align}
This can be described as a Cram\'{e}r-von Mises change-point statistic. In the $\mathds R^d$-valued case, the weight function is a positive function $ w \colon \mathds R^d \to \mathds R$ with
\begin{align*}
\int _{\mathds R^d} w(t) \ dt < \infty.
\end{align*}
The empirical process has been bootstrapped by \cite{Buel} and \cite{NaRa} and recently by \cite{DoLaLeNe} using the wild bootstrap and by \cite{KoYa} using the weighted bootstrap. \\
Our bootstrapped version of (\ref{UnsereTestStatistikVerteilung}) is
\begin{align}
T_{n,w}^{\ast} = \max_{ 1 \leq m \leq kp-1} \frac{1}{kp} \int_{\mathds R^d} \left( \sum_{i=1}^m 1_{ \{X_i^{\ast} \leq t \} } - \frac{m}{kp} \sum_{i=1}^{kp} 1_{ \{X_i^{\ast} \leq t \} } \right)^2 w(t) \ dt , \label{UnsereTestStatistikVerteilungBoots}
\end{align}
where the sample $X_1^{\ast}, \dots, X_{kp}^{\ast}$ is produced by the non-overlapping block bootstrap.\\
We will now state conditions, under which the bootstrap method is justified.

\begin{Cor} \label{LimitTheoremDistributionalChange} 
Let $(X_n)_{n \in \mathds N}$ be $\mathds R ^d$ valued random variables, $L_1$- near epoch dependent on a stationary, absolutely regular sequence $(\xi_n)_{n \in \mathds Z}$, such that for some $\delta > 0$ 
\begin{enumerate}
\item [1.] $ \sum_{m=1} ^{\infty} m^2 (a_m)^{\delta / (1+ 4 \delta)} < \infty$,
\item [2.] $ \sum_{m=1}^{\infty} m^2 (\beta_m)^{\delta / (\delta+3)} < \infty$.
\end{enumerate} 
Let the block length $p$ be nondecreasing with $p(n) = O(n^{1- \epsilon})$ for some $\epsilon >0$ and $p(n) = p(2^l)$ for $n=2^{l-1} +1 , \dots ,2^l$. \\
Then, almost surely, the conditional distribution of $T_{n,w}^{\ast}$, given $X_1, \dots, X_n$, converges to the same limit as the distribution of $T_{n,w}$, as $n \to \infty$.
\end{Cor}

Note, that producing a bootstrap sample $X_1^{\ast}, \dots , X_{kp}^{\ast}$ first, and then treating the indicators 
\begin{align*}
1_{ \{ X_1^{\ast} \leq \cdot \} }, \dots , 1_{ \{ X_{kp}^{\ast} \leq \cdot \} },
\end{align*}
is the same as if we first look upon the indicators as $H$-valued random variables $Y_1, \dots , Y_n$ and then generate $Y_1^{\ast}, \dots , Y_{kp}^{\ast}$. \\
Now we can apply Corollary \ref{LimitTheoremCusumHilbertraum} and therefore we have to verify the conditions of Theorems \ref{FCLTHilbertraum} and \ref{SeqBootstrapHilbertraum}, respectively. \\
The moment condition is automatically satisfied, due to the definition of $w(t)$ and the dependence conditions are satisfied because of Lemma 2.2 in Dehling, Sharipov and Wendler \cite{DeShWe} and the Lipschitz-continuity of the mapping $x \mapsto 1_{ \{ x \leq \cdot \} }$.

Note, that producing a bootstrap sample $X_1^{\ast}, \dots , X_{kp}^{\ast}$ first, and then analyzing the indicators 
\begin{align*}
1_{ \{ X_1^{\ast} \leq \cdot \} }, \dots , 1_{ \{ X_{kp}^{\ast} \leq \cdot \} },
\end{align*}
is the same as if we first look upon the indicators as $H$-valued random variables $Y_1, \dots , Y_n$ and then generate $Y_1^{\ast}, \dots , Y_{kp}^{\ast}$. \\
We can apply Corollary 1, if we can verify the conditions of Theorems \ref{FCLTHilbertraum} and \ref{SeqBootstrapHilbertraum}, respectively. The moment condition is automatically satisfied, due to the definition of $w(t)$, and the dependence conditions are satisfied because of Lemma 2.2 in \cite{DeShWe} and the Lipschitz-continuity of the mapping $x \mapsto 1_{ \{ x \leq \cdot \} }$.

\begin{Rem}
Instead of the inner product we have defined one can use
\begin{align*}
\langle f,g \rangle _{1} = \int f(t)g(t) \ dt \ \ \ \ \text{or} \ \ \ \ \langle f,g \rangle _{dF} = \int f(t)g(t) \ dF(t),
\end{align*}
which lead to  well known change point statistics. Note that in the first case the norm of the indicator $1_{ \{ X_1 \leq \cdot \} }$ is infinite, which is remedied by considering $1_{ \{ X_1 \leq \cdot \} } - F(\cdot)$. Additional moment assumptions on the $X_i$ may be needed to make Corollary 3 hold also in this case.
\end{Rem}

\section{Real-life data examples} \label{SecRealLife}

\begin{figure}[h!] 
\centering
\includegraphics [width=\textwidth]{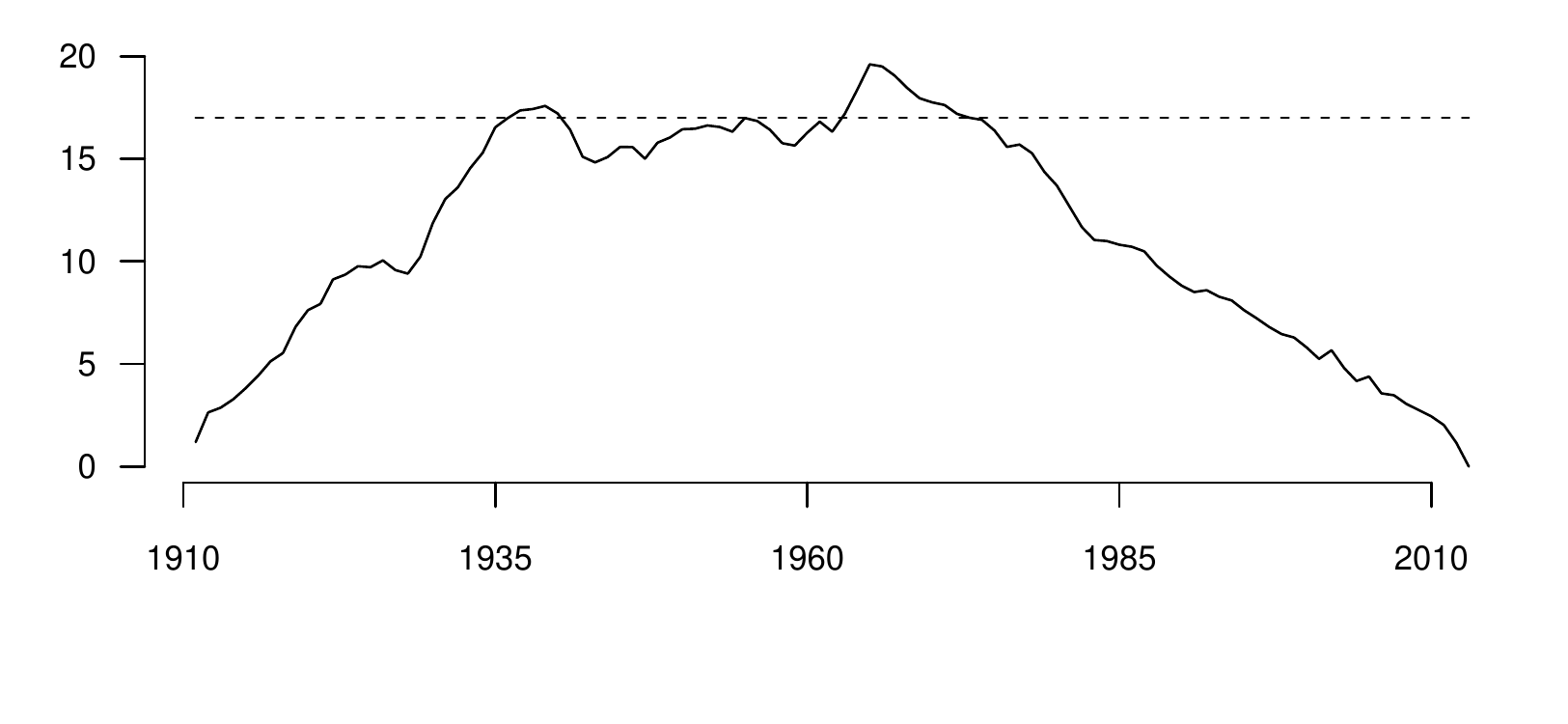}
\caption{Process $\frac 1 {\sqrt n} \lVert \bar X_k - k/n \bar X_n \rVert$ (black line) computed from $103$ annual flow curves of the river Chemnitz and $0.95$ level of significance (dashed line) computed from $999$ bootstrap iterations.}
\label{Abb1}
\end{figure}

\begin{figure}[h!]
\centering
\includegraphics [width=\textwidth]{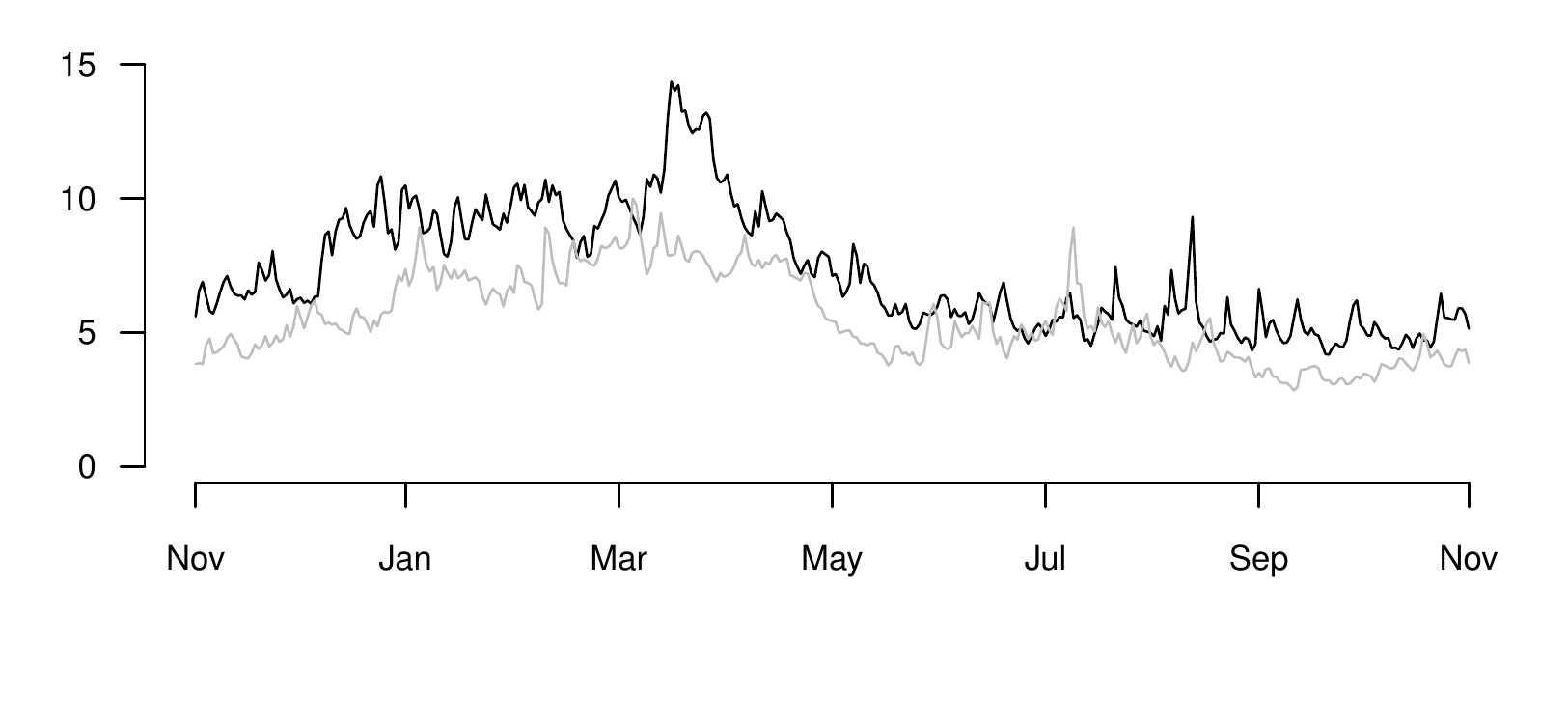}
\caption{Average annual flow curves of the time period 1910 - 1964 (grey line) and the time period 1965 - 2012 (black line).}
\label{Abb2}
\end{figure}

To illustrate our methods we apply the tests, described in the previous subsections, to hydrological observations. \\
The first data set contains average daily flows of the river Chemnitz at Goeritzhain for the time period 1910 - 2012. Thus one gets 103 annual flow curves which can be interpreted as realizations of $\mathds R^{365}$-valued random variables that are dependent over time. Alternatively one could smoothen the curves and hence get functional data. 
Let $X_i$ be the $i$th annual curve, taking its value in $\mathds R^{365}$. Figure \ref{Abb1} shows the process 
\begin{align*}
 \frac{1}{\sqrt n} \left \lVert \sum_{i=1}^k X_i - \frac{k}{n} \sum_{i=1}^n X_i \right \rVert   \ \ \ \ k=1, \dots n-1.
\end{align*}
 The value of the test statistic is the maximum of this process, which is attained in 1964. Because it is larger than the bootstrapped $5 \%$ level of significance, the test indicates that there has been a change in structure of the annual flow curves. \\
Figure \ref{Abb2} illustrates the character of this change by comparing the average flow curves based on the data before and after 1964. \\
Of course there are other methods to deal with this data set. One might adopt the methodology of \cite{RoLuGaLu}, used to detect changes in storm frequency and strengths. Here one might jointly test for changes in the yearly flood counts and the corresponding river heights. 
\\\\
As a second example, we look at annual maximum flows (the flows are annual maximums over daily observations) of the river Elbe at Dresden for the time period 1850 - 2012, see figure \ref{Abb3a}. 

\begin{figure}[h!]
\centering
\includegraphics [width=\textwidth]{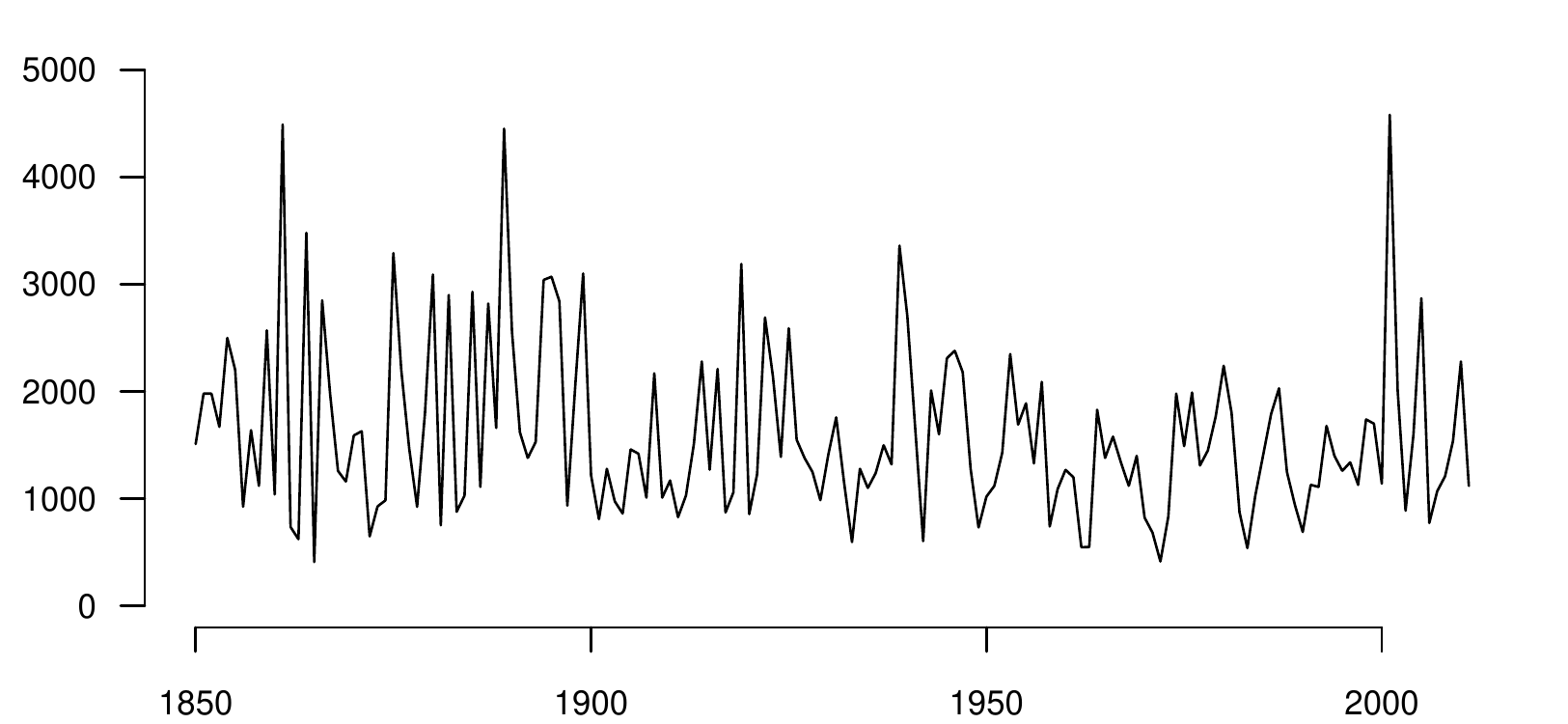}
\vspace{20pt}
\caption{Annual maximum flows of the river Elbe at Dresden from 1850 to 2012.}
\label{Abb3a}
\end{figure}

\begin{figure}[h!]
\centering
\includegraphics [width=\textwidth]{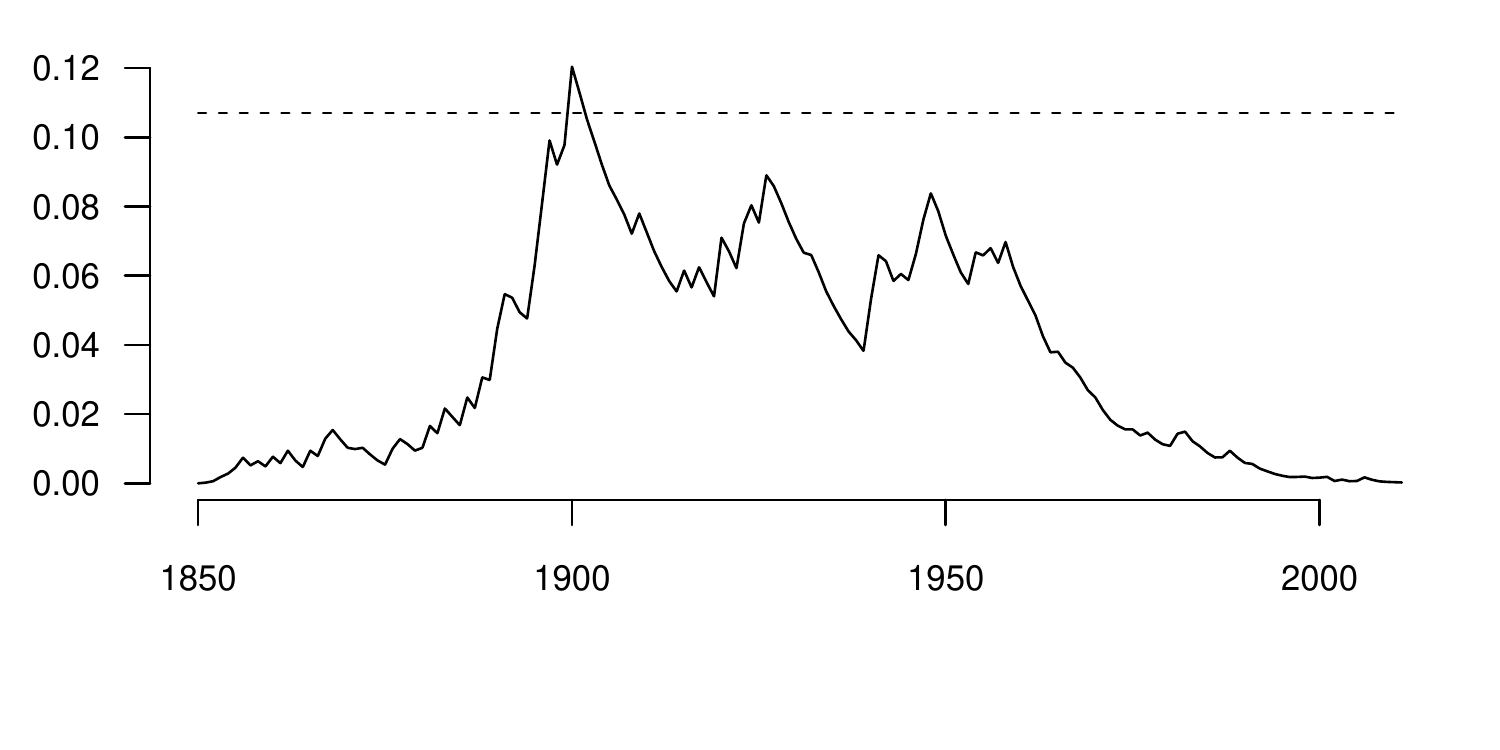}
\caption{Process $\frac 1 n \int (\hat F_k(x) -k/n \hat F_n(x)^2  \phi (x) dx$ (black line) computed from $163$ annual maximum flows of the river Elbe and $0.95$ level of significance (dashed line) computed from $999$ bootstrap iterations.}
\label{Abb3}
\end{figure}

\begin{figure}[h!]
\centering
\includegraphics [width=\textwidth]{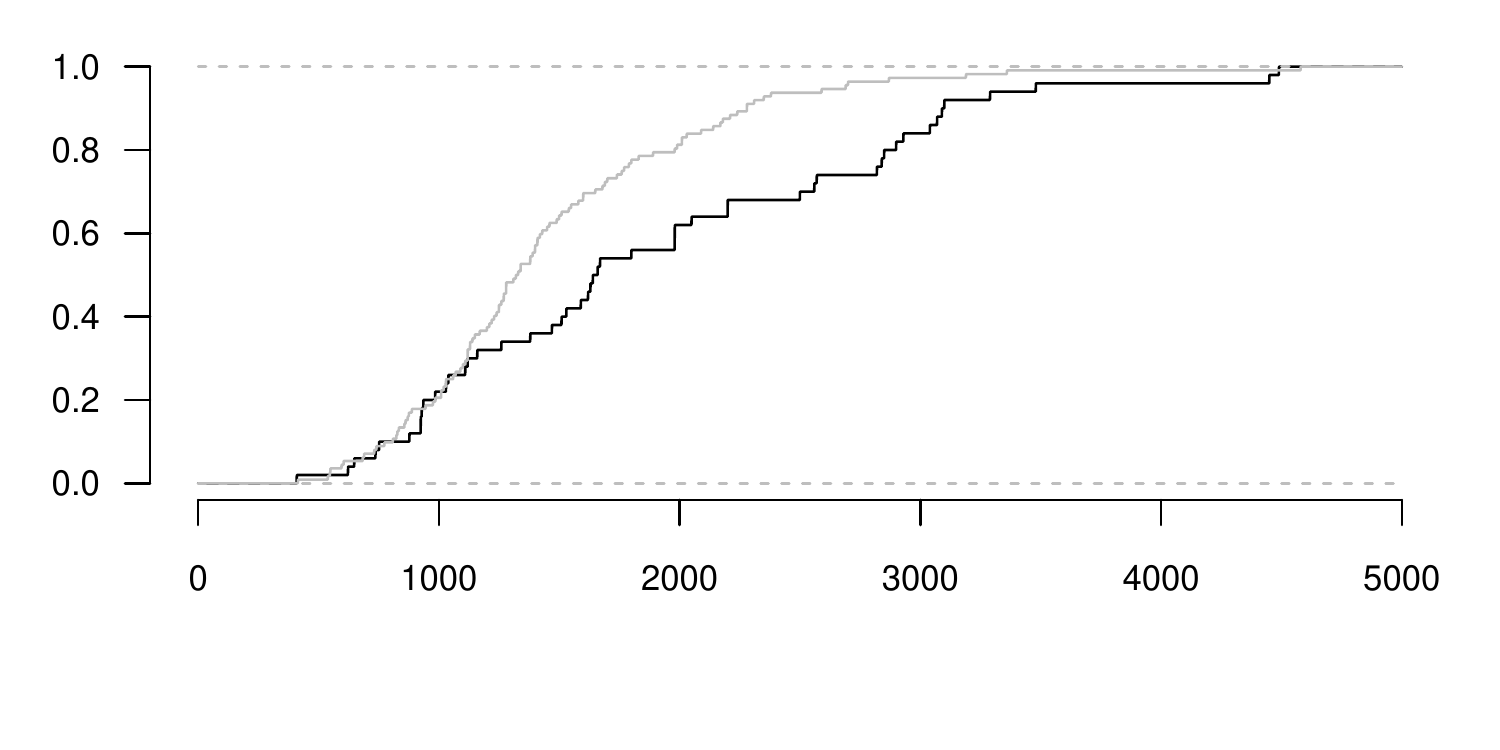}
\caption{Empirical distribution functions of the first $50$ observations (black line) and the last $113$ observations (grey line).}
\label{Abb4}
\end{figure}

In the statistical analysis of floods annual maxima are typically modeled as independent. However, such time series often display some correlation in truth. Classical methods of extreme value theory sometimes fail if observations are dependent, a problem that is bypassed by our method. Moreover, the data seem to have heavy tails. But Corollary \ref{LimitTheoremDistributionalChange} does not require any moment conditions and hence we may apply the test for distributional change to these $\mathds R$-valued observations. For this purpose compute (\ref{UnsereTestStatistikVerteilung}) and $999$ iterations of (\ref{UnsereTestStatistikVerteilungBoots}). Figure \ref{Abb3} shows the process 
\begin{align*}
\frac{1}{n}\int \left (\sum_{i=1}^k 1_{ \{ X_i \leq x\} } - \frac{k}{n} \sum_{i=1}^n 1_{ \{ X_i \leq x\} } \right )^2 \phi(x) dx \ \ \ \ k=1, \dots , n-1
\end{align*}
where we have used the probability density of the $N(2000,2000^2)$ distribution as weight function $\phi(\cdot)$. The value of the test statistic equals the maximum of this process, which is larger than the bootstrapped level of significance and therefore a change is detected. \\
Finally Figure \ref{Abb4} compares the empirical distribution functions based on the data before and after 1900, which is where the maximum is attained. The comparison indicates that moderately severe floods have become less frequent.

\section{Simulation Study}

\subsection{CUSUM test for functional data}

\begin{table}
\centering
\caption{Empirical size for CUSUM test for the FAR(1)-model with Gaussian-/ Wiener-kernel; nominal size $\alpha=0.1$}
{\footnotesize 	\begin{tabular*}{\columnwidth}{@{\extracolsep{\fill} } c c c c c c c c}
\toprule[1pt]
&  $ \lVert \psi \rVert_{L^2} $ & $0$ & $ 0.1$ & $ 0.2$ & $0.4$ & $0.6$ \\
\midrule [1pt]
& $p=4$ & $0.110/ 0.106$ & $0.118/ 0.122$ &  $0.138/ 0.154$ & $0.171/ 0.189$ & $0.274/ 0.263$  \\
& $p=5$ &$0.098/ 0.117$ & $0.106/ 0.124$ & $0.131/ 0.099$ & $0.159/ 0.149$ & $0.214/ 0.235$ \\
$n=50$	& $p=6$ & $0.138/ 0.111$	 & $0.120/ 0.133$ & $0.143/ 0.120$ & $0.151/ 0.168$ & $0.193/ 0.216$  \\
& $p=7$ & $0.121/ 0.110$ & $0.117/ 0.110$ & $0.149/ 0.116$ & $0.153/ 0.142$ & $0.209/ 0.203$ \\
& $p=8$& $0.104/ 0.111$ & $0.139/ 0.125$ & $0.157/ 0.144$ & $0.164/ 0.146$ & $0.195/ 0.211$ \\
\midrule[0.5pt] 
& $p=6$& $0.107/ 0.107$& $0.102/ 0.123$& $0.139/ 0.131$& $0.157/ 0.177$& $0.236/ 0.222$\\
& $p=8$& $0.121/ 0.097$& $0.118/ 0.119$& $0.117/ 0.133$& $0.162/ 0.136$& $0.199/ 0.187$\\
$n=100$ & $p=10$ & $0.105/ 0.111$& $0.107/ 0.113$& $0.114/ 0.116$& $0.123/ 0.128$& $0.151/ 0.170$ \\
& $p=11$ & $0.112/ 0.089 $& $0.128/ 0.114$& $0.104/ 0.131$& $0.135/ 0.120$& $0.170/ 0.169$\\
& $p=13$	& $0.123/ 0.132$& $0.138/ 0.132$& $0.146/ 0.133$& $0.178/ 0.183$& $0.204/ 0.201$ \\
\bottomrule[1pt] 
\end{tabular*}
}
\label{Tab01}
\end{table}

In this simulation study we will apply our CUSUM test to realizations of functional time series, given by
\begin{align*}
Y_i (t) = \begin{cases}
     X_i(t) & \text{for } i  \leq k \\
     X_i(t)  + \mu(t) & \text{for } i > k,
     \end{cases}
\end{align*}
for $t \in [0,1]$. The function $\mu \colon [0,1] \to \mathds R$ describes the change, $k$ is the time of the change and $(X_i(t))_{i \geq 1}$ is a functional, weakly dependent sequence. As model for this sequence we will use functional autoregressive processes of order $1$ (FAR(1)), formally
\begin{align}
X_i(t) = \int_0^1\psi (t,s) X_{i-1}(s) \ ds + \epsilon_i(t), \label{RotesModellFAR}
\end{align}
see \cite{Bosq}. The $(\epsilon_i(t))_{i \geq 1}$ are independent and Gaussian and $\psi(s,t)$ is a kernel function, satisfying
\begin{align*}
\lVert \psi_G \rVert_{L^2([0,1]^2)}^2 = \int_0^1 \int_0^1 \psi^2(s,t) \ ds \ dt < 1.
\end{align*}
As kernel functions we use
\begin{align*}
\psi_G(s,t) = C_1 \exp ( (s^2+t^2)/2) \ \ \ \ \text{or}  \ \ \ \ \psi_W(s,t) = C_2 \min(s,t),
\end{align*}
the so-called Gaussian- or Wiener kernels, respectively. One obtains 
\begin{align*}
\lVert \psi_G \rVert_{L^2([0,1]^2)} \approx  C_1 (0.6832)^{-1} \ \ \ \ \text{and} \ \ \ \ \lVert \psi_W \rVert_{L^2([0,1]^2)} = C_2 6^{-1/2}.
\end{align*}

\begin{table}
\centering
\caption{Empirical power for CUSUM test for the FAR(1)-model with Gaussian-/ Wiener-kernel; $\mu(t) = \sin (t)$, change after $50\%$ of the observations; nominal size $\alpha=0.1$}
{\footnotesize 	\begin{tabular*}{\columnwidth}{@{\extracolsep{\fill} } c c c c c c c c}
\toprule[1pt]
&  $ \lVert \psi \rVert_{L^2} $ & $0$ & $ 0.1$ & $ 0.2$ & $0.4$ & $0.6$ \\
\midrule [1pt]
& $p=4$ & $1.000/ 1.000$ & $1.000/ 1.000$ & $0.997/ 1.000$ & $0.988/ 0.986$ & $0.927	/ 0.894$ \\
& $p=5$ & $1.000/ 1.000$ & $0.999/ 0.999$ & $0.998/ 0.999$ & $0.975/ 0.972$ & $0.861	/ 0.850$ \\
$n=50$ & 	$p=6$ & $1.000/ 1.000$ & $1.000/ 0.999$ & $0.996/ 0.998$ & $0.968/ 0.971$ & $0.859/ 0.830$ \\
& $p=7$ & $1.000/ 1.000$ & $1.000/ 1.000$ & $0.996/ 0.999$ & $0.969/ 0.957$ & $0.858/ 0.828$ \\
& $p=8$ & $0.998/ 0.998$ & $0.998/ 0.998$ & $0.989/ 0.990$ & $0.961/ 0.949$ & $0.808/ 0.808$ \\
\bottomrule[1pt] 
\end{tabular*}
}
\label{Tab02}
\end{table}

Note that the $L^2$-norms of the kernel functions cause the strength of the dependence in the sequences.  \\
In the simulation study we have reproduced the implementation mode of \cite{Tor2}, using the R-package \texttt{fda}. The $\epsilon_i(t)$ are created from Brownian bridges, which are then transformed to functional data objects by the R-function \texttt{Data2fd}, using $25$ B-spline functions. We set $X_{-99}(t) = \epsilon_{-99}(t)$ and $X_i$ as in (\ref{RotesModellFAR}) for $i \geq -98$. Using a burn-in period of length $100$ we discard $X_{-99}, \dots X_0$. Afterwards a function $\mu(t)$ is eventually added to $X_k(t), \dots, X_n(t)$, describing the structural change. Finally the CUSUM test is applied to these sequences, where critical values are obtained from $J=499$ bootstrap iteration. Moreover empirical size and empirical power are deduced from $1000$ simulation runs. \\
Table \ref{Tab01} shows the empirical size (empirical probabilities that the hypothesis is rejected) of the test. For almost all combinations of dependencies and block lengths it is higher then the nominal size. However, as long as the dependence is not to strong ($ \lVert \psi \rVert \leq 0.2$ for $n=50$, $\lVert \psi \rVert \leq 0.4$ for $n=100$) this happens to an acceptable degree. For $\lVert \psi \rVert =0.6$ the probability of a type I error becomes to high. This is hardly surprising, as bootstrapped test suffer from this issue even if the observations are real valued, see table \ref{Tab1} below. Finally one might compare the outcome of the test for the different FAR(1)-models. The test performs better if functional observations are generated using the Gaussian kernel, but only for the right choice of block length. \\
Table \ref{Tab02} shows the empirical power of the test. We consider the same alternative as \cite{Tor2}, that is $\mu(t) = \sin (t)$ and $k=\lfloor n/2 \rfloor$. The power is very good and decreases only slightly as the dependence grows. \\
An alternative to our test is the method of \cite{BeGaHoKo}, using functional principal components. The finite dimensional behavior of this test for dependent data was investigated in \cite{Tor2}, using the FAR(1)-model. In his simulation study the empirical size is clearly beneath the nominal size. Depending on the choice of projection dimension and the selection of the bandwidth for variance estimation, the empirical power might vanish. In contrast our test has good power properties for all block lengths.

\subsection{Cram\'{e}r-von Mises change-point test}

\begin{table}
\centering
\caption{Empirical size for Cram\'{e}r-von Mises/CUSUM test; nominal size $\alpha=0.05$}
{\footnotesize 	\begin{tabular*}{\columnwidth}{@{\extracolsep{\fill} } c c c c c }
\toprule[1pt]
 &   & $a_1= 0.2$ & $a_1=0.5$ & $a_1=0.8$ \\
\midrule [1pt]
& $p=4$ & $0.05/ 0.064$ & $0.127/ 0.14$ & $0.23/ 0.318$ \\
$n=50$ & $p=5$ & $0.044/ 0.063$& $0.085/ 0.097$ & $0.212/ 0.226$ \\
& $p=7$ & $0.046/ 0.051$ & $0.076/ 0.075$ & $0.155/0.145$ \\
\cmidrule[0.5pt] {2-5}
& $p=6$ & $0.035/ 0.078$ & $0.082/0.111$ &  $0.254/0.225$\\
$n=100$ & $p=8$ & $0.056/0.062$ & $0.059/0.079$ & $0.171/0.173$ \\
 & $p=10$ & $0.047/0.048$ & $0.072/0.054$ & $0.131/0.123$ \\
  & $p=12$ &  0.056/0.064&  0.074/0.085& $0.122/0.126$ \\
\cmidrule[0.5pt] {2-5}
& $p=8$ & $0.061/0.078$ & $0.091/0.061$ & $0.201/0.208$ \\
$n =200$  & $p=10$ & $0.04/0.061$ & $0.064/0.085$ & $0.149/0.156$ \\
 & $p=12$ & $0.055/0.061$ & $0.067/0.057$ & $0.137/0.0140$ \\
 & $p=15$ & $0.042/0.057$ & $0.066/0.063$ & $0.1/0.104$ \\
\bottomrule[1pt] 
\end{tabular*}
}
\label{Tab1}
\end{table}

\begin{table}
\centering
\caption{Empirical power for Cram\'{e}r-von Mises/CUSUM test; change of height $\mu$ after $50\%$ of the observations; nominal size $\alpha=0.05$}
{\footnotesize
\begin{tabular*} {\columnwidth}{@{\extracolsep{\fill} } c c c c }
\toprule[1pt]
\multicolumn{4}{c}{$\mu =0.5$} \\
\midrule [1pt]
 &  $a_1=0.2$ & $a_1=0.5$ & $a_1=0.8$ \\
\midrule [1pt]
 & $p=4$ & $p=7$ & $p=7$  \\
$n=50$ & $0.23/0.233$ & $0.161/0.156$& $0.207/0.194$  \\
 & $p=10$ & $p=8$ & $p=12$ \\
$n=100$ & $0.302/0.315$ & $0.28/0.262$ & $0.206/0.206$  \\
 & $p=12$ & $p=12$ & $p=15$  \\
$n =200$  & $0.669/0.7$ & $0.456/0.462$ & $0.258/0.295$  \\
\midrule [1pt]
\multicolumn{4}{c}{$\mu =1$} \\
\midrule [1pt]
 & $p=4$ & $p=7$ & $p=7$  \\
$n=50$ & $0.686/0.678$ & $0.313/0.431$& $0.351/0.335$  \\
 & $p=10$ & $p=8$ & $p=12$ \\
$n=100$ & $0.847/0.851$ &  $0.695/0.708$ & $0.419/0.373$  \\
 & $p=12$ & $p=12$ & $p=15$  \\
$n =200$  & $0.995/0.998$ & $0.937/0.945$ & $0.64/0.630$  \\
\bottomrule[1pt] 
\end{tabular*}
}
\label{Tab2}
\end{table}

In a second simulation study we investigate the finite sample performance of the Cram\'{e}r-von Mises-type change-point test, which compares the empirical distribution functions. We are considering different block lengths $p$ and three kinds of dependencies. The data generating process is an AR$(1)$-process satisfying
\begin{align*}
X_i = a_1 X_{i-1} + \epsilon_i,
\end{align*}
with $a_1 \in \{ 0.2,0.5,0.8\}$ and $ (\epsilon_i)_{i \geq 1}$ iid with $\epsilon_i \sim N(0,1-a_1^2)$. In all situations we have calculated critical values from $J=999$ bootstrap-iterations and empirical size and power from $m=1000$ iterations of the test. In addition we have applied the classical CUSUM test to the data, which compares sample means. For the execution of this test see subsection \ref{ChangeInMeanSection} and consider the special case $H= \mathds R$. The number of bootstrap-iterations is set to $999$, too. \\
Table \ref{Tab1} reports empirical sizes under the hypothesis of no change. For the low correlation case ($a_1=0.2$) the performance is quite good, even for small sample sizes like $n=50$. When $a_1=0.8$ the empirical size is drastically larger than the nominal one. This is typical for bootstrap tests due to an underestimation of covariances, see for example \cite{DoLaLeNe}. Altogether there are only marginal differences between Cram\'{e}r-von Mises and CUSUM test. Note that for the different tests, different choices of block length are advantageous. \\
Regarding the power of our test we choose for each sample size and AR-coefficient the block length that provides the best empirical size under this circumstances. We start with the following change-in-mean model:
\begin{align*}
Y_i = \begin{cases}
     X_i & \text{for } i  \leq k \\
     X_i + \mu & \text{for } i > k. 
     \end{cases}
\end{align*}
Table \ref{Tab2} gives an overview of the empirical power under this alternative for $\mu =0.5$ and $\mu =1$, respectively. We see that a level shift of height $\mu=0.5$ in an AR-process with $a_1=0.8$ is to small to be detected. However for larger shifts ($\mu =1$) the power of our test is notably good.\\
The CUSUM test is designed to detect changes in the mean. If critical values can be deduced from a known asymptotic distribution, the CUSUM test is supposed to have greater power then our test. However if critical values are investigated by the bootstrap, table \ref{Tab2} indicates that both tests have similar power properties.\\
To illustrate the power of our test against several alternatives, consider a change in the skewness of a process. Therefore we need a second data generating process $X_i'=a_1X_{i-1}'+\epsilon_i'$, independent of the first one, and define
\begin{align*}
Y_i = \begin{cases}
     X_i^2 + X_i'^2 & \text{for } i  \leq k \\
     4- (X_i^2+X_i'^2) & \text{for } i > k. 
     \end{cases}
\end{align*}
Table \ref{Tab4} shows that against this alternative the power of the Cram\'{e}r-von Mises test is excellent for $n=200$ and coefficients $a_1 \leq 0.5$. The same table illustrates the power of the CUSUM test. Apparently this test does not see changes in the skewness when the mean is unmodified. \\
To summarize, the Cram\'{e}r-von Mises test can be used as an omnibus test for change in the marginal distribution without prespecifying the type of a change. In the case of a change in mean, the power is not much lower compared to the classical CUSUM test. Therefore the test that is based on the Cram\'{e}r-von Mises statistic seems advantageous.

\begin{table}
\centering
\caption{Empirical power for Cram\'{e}r-von Mises/CUSUM test; change in skewness after $50\%$ of the observations; nominal size $\alpha=0.05$}
{\footnotesize
\begin{tabular*} {\columnwidth}{@{\extracolsep{\fill} } c c c c }
\toprule[1pt]
 &  $a_1=0.2$ & $a_1=0.5$ & $a_1=0.8$ \\
\midrule [1pt]
 & $p=4$ & $p=7$ & $p=7$  \\
$n=50$ & $0.323/0.047$ & $0.244/0.062$& $0.196/0.08$  \\
 & $p=10$ & $p=8$ & $p=12$ \\
$n=100$ & $0.546/0.033$ & $0.461/0.045$ & $0.223/0.076$  \\
 & $p=12$ & $p=12$ & $p=15$  \\
$n =200$  & $0.945/0.04$ & $0.846/0.045$ & $0.375/0.065$  \\
\bottomrule[1pt] 
\end{tabular*}
}
\label{Tab4}
\end{table}

\section*{Acknowledgements}
The research was supported by the DFG Sonderforschungsbereich 823 (Collaborative Research Center) Statistik Nichtlinearer Dynamischer Prozesse. \\
The authors are grateful to Svenja Fischer and Andreas Schumann from the faculty of civil engineering, Ruhr-Universit\"{a}t Bochum, for providing hydrological data.

\bibliography{Lit}

\appendix

\section{Prelimanary Results}

\begin{The}\label{StraffheitKritBillingsley}
Let $\{W_n\}_{n \geq 1}$ be a sequence of $D_H[0,1]$- valued random functions with $W_n(0) = 0$. Then $\{W_n\}_{n \geq 1}$ is tight in $D_H[0,1]$ if the following condition is satisfied:
\begin{align*}
\lim_{\delta \to 0} \limsup_{n \to \infty} \frac{1}{\delta} P \left( \sup_{s \leq t \leq s +\delta} \lVert W_n(t) \rVert > \epsilon \right) = 0,
\end{align*}
for each positive $\epsilon$ and each $s \in  [0,1]$. \\
Furthermore the weak limit of any convergent subsequence of $\{W_n\}$ is in $C_H [0,1]$, almost surely.
\end{The}

For real valued random variables this is Theorem 8.3 of \cite{Bil}, which carries over to $D[0,1]$. The proof still holds for H-space valued functions.
\\\\
The next lemma is Lemma 4.1 of \cite{ChWh} with the slight modification that the third condition contains fourth moments instead of second moments. Let $(e_i)_{i \geq 1}$ be an orthonormal basis of $H$. Then $H_k$ is the closed linear span of $(e_i)_{1 \leq i \leq k}$ and $P_k \colon H \to H_k$ the projection operator.

\begin{Lem} \label{ConditionWeakConvHilbertspace}
Let $\{W_n\}_{n \geq 1}$ be a sequence of $D_H[0,1]$- valued random functions. Let $W^k$ be a Brownian motion in $H_k$ with  $S^k$ being the covariance operator of $W^k(1)$. Suppose the following conditions are satisfied:
\begin{enumerate}
\item [(i)] For each $k \geq 1$, $P_k W_n \Rightarrow W^k$ in $D_{H_k}[0,1]$ (as $n \to \infty$),
\item [(ii)] $W^k \Rightarrow W$ in $D_{H}[0,1]$ (as $k \to \infty$),
\item [(iii)] $\limsup_{n \to \infty} E \left( \sup_{t \in [0,1]} \lVert W_n(t) - P_kW_n(t)\rVert ^4 \right) \rightarrow 0$ as $k \to \infty$.
\end{enumerate}
Then $W_n \Rightarrow W$ in $D_H[0,1]$, where $W$ is a Brownian motion in $H$ with covariance operator $S$.
\end{Lem}

\begin{Lem} \label{4teMomentenungleichung}
Let $(X_n)_{n \geq 1}$ be $H$-valued, stationary and $L_1$-near epoch dependent on an absolutely regular process with mixing coefficients $(\beta(m))_{m \geq 1}$ and approximation constants $(a_m)_{m \geq 1}$. If $EX_1 =0$ and
\begin{enumerate}
\item [(i)] $E \lVert X_1 \rVert^{4 + \delta} < \infty$,
\item [(ii)] $ \sum_{m=1}^{\infty} m^2 (a_m)^{\delta / (\delta+3)} < \infty$,
\item [(iii)] $ \sum_{m=1}^{\infty} m^2(\beta(m))^{\delta / (\delta +4)} < \infty$,
\end{enumerate}
holds for some $\delta > 0$, then
\begin{align*}
E \lVert X_1 + X_2 + \dots + X_n \rVert ^4 \leq C n^2 \left( E \lVert X_1 \rVert ^{4 + \delta} \right)^{\frac{1}{1+ \delta}}.
\end{align*}
\end{Lem}

The result follows from the proof of Lemma 2.24 of \cite{BoBuDe}, which is also valid for Hilbert spaces.

\begin{Lem} \label{Maximalungleichung}
Let $(X_n)_{n \geq 1}$ be a stationary sequence of $H$-valued random variables such that $EX_1 =0$, $E \lVert X_1 \rVert ^4 < \infty$ and for some $C> 0$
\begin{align*}
E \lVert X_1 + X_2 + \dots + X_n \lVert ^4 \leq C n^2 .
\end{align*}
Then
\begin{align*}
E  \max_{k \leq n} \lVert X_1 + X_2 + \dots + X_k \lVert ^4 \leq C n^2 .
\end{align*}
\end{Lem}

This Lemma is a special case of Theorem 1 of \cite{Mor}. The proof carries over directly to Hilbert spaces.

\section{Proofs of the main results}

\begin{proof} [Proof of Theorem \ref{FCLTHilbertraum}] 
We will prove the theorem by verifying the three conditions of Lemma A1. To show (i) we start with the special case $H= \mathds R$. Let $EX_1 =0$. Then by Lemma 2.23 of \cite{BoBuDe} we have
\begin{align}
 \nonumber \frac{1}{n} E \left ( \sum_{i=1}^n X_i \right)^2 \rightarrow \sigma ^2,
\end{align}
where $\sigma^2 = \sum_{i = - \infty}^{\infty} E(X_0X_i)$ and this series converges absolutely. \\
Furthermore by Lemma 2.4. of \cite{DeShWe} we have
\begin{align}
\frac{1}{\sqrt n} \sum_{i=1}^n (X_i- \mu) \Rightarrow N (0, \sigma ^2). \label{ReellConvergencePartialSum}
\end{align}
In order to show convergence of the finite dimensional distributions of $W_n(t)=n^{-1/2}\sum_{i=1}^{\lfloor nt \rfloor} X_i$, we will show
\begin{align}
\left (W_n(t),W_n(1)-W_n(t) \right ) \Rightarrow \left (\sigma W(t), \sigma(W(1)-W(t)) \right ), \label{fidiReell}
\end{align}
where $W$ is standard Brownian motion in $\mathds R$. This can be easily adopted to higher dimensions than $2$. Remember that $(X_i)_{i\geq 1}$ is $L_1$-near epoch dependent on an absolutely regular process $(\epsilon_i)_{i \in \mathds Z}$ and $ \mathcal F_{-l}^m= \sigma(\epsilon_{-l}, \cdots , \epsilon_m)$. We proceed as in the proof of Theorem 21.1 in \cite{Bil}. Define
\begin{align*}
& U_{n, \lfloor n t \rfloor} = \frac 1 {\sqrt n} E \left [ \sum_{i=1}^{\lfloor nt \rfloor - 2j_n} X_i \ \vert \ \mathcal F_{- \infty}^{\lfloor nt \rfloor-j_n} \right ] \\
\text{and} \ \ \ \ & V_{n, \lfloor n t \rfloor} = \frac 1 {\sqrt n} E \left [ \sum_{i=1+\lfloor n t \rfloor + 2 j_n}^{n} X_i \ \vert \ \mathcal F^{\infty}_{\lfloor nt \rfloor+j_n} \right ],
\end{align*}
for positive integers $j_n \to \infty$. \cite{Bil} shows
\begin{align}
\lvert U_{n,\lfloor nt \rfloor} - W_n(t) \rvert \xrightarrow{\mathcal P} 0 \ \ \text{and} \ \ \lvert V_{n,\lfloor nt \rfloor} - (W_n(1) - W_n(t)) \rvert \xrightarrow{\mathcal P} 0, \label{ApproxBedErwartung}
\end{align}
and thus by (\ref{ReellConvergencePartialSum}) and Slutsky's theorem we obtain for fixed $t$ 
\begin{align}
U_{n,\lfloor nt \rfloor} \Rightarrow \sigma W(t) \ \ \text{and} \ \ V_{n, \lfloor nt \rfloor}  \Rightarrow \sigma (W(1) -W(t)). \label{EinzelneKonvergenzApprox}
\end{align}
Further for all Borel sets we get by definition of $U_{n,{\lfloor nt \rfloor}}$ and $V_{n,{\lfloor nt \rfloor}}$
\begin{align}
\nonumber \left \lvert P \left ( U_{n,\lfloor nt \rfloor} \in H_1, V_{n,\lfloor nt \rfloor} \in H_2 \right ) - P\left ( U_{n,\lfloor nt \rfloor} \in H_1\right ) P \left( V_{n,\lfloor nt \rfloor} \in H_2 \right ) \right \rvert \\
\leq  \alpha( \mathcal F_{-\infty}^{\lfloor nt \rfloor -j_n}, \mathcal F^{\infty}_{\lfloor nt \rfloor+j_n} ) \leq \alpha(j_n) \rightarrow 0, \label{StrongMixingRot}
\end{align}
as $n \to \infty$, where $\alpha()$ is the strong mixing coefficient. $\alpha(j_n)$ converges to $0$ because the $(\epsilon_i)_{i \in \mathds Z}$ are absolute regular and this implies strong mixing. For the definition of strong mixing see for example \cite{ChWh}. Combining (\ref{EinzelneKonvergenzApprox}) with (\ref{StrongMixingRot}) we arrive at
\begin{align*}
\left (U_{n,\lfloor nt \rfloor}, V_{n, \lfloor nt \rfloor} \right ) \Rightarrow \left (\sigma W(t), \sigma (W(1) -W(t)) \right ),
\end{align*}
where weak convergence takes place in $\mathds R^2$. However, because of (\ref{ApproxBedErwartung}) this implies (\ref{fidiReell}).
If we can show that the set
\begin{align}
\left \{ \max_{ s \leq t \leq s + \delta} \frac{1}{\delta} (W_n(t) - W_n(s))^2 \ \vert \ 0 \leq s \leq 1, 0 \leq \delta \leq 1, n \leq N(s, \delta) \right \} \label{GlgrIntbar}
\end{align}
is uniformly integrable, then according to Lemma 2.2 in \cite{WoWh} $W_n$ is tight in $D[0,1]$ equipped with the Skorohod topology. Furthermore the weak limit is almost surely in $C[0,1]$. \\
So fix $s \in [0,1]$ and $\delta \in [0,1]$. By the proof of Lemma 2.24 in \cite{BoBuDe} we obtain
\begin{align*} 
E \left ( \sum_{i = \lfloor n s \rfloor +1}^{\lfloor n(\delta +s) \rfloor} X_i \right )^4 \leq C (\lfloor n(\delta + s) \rfloor - \lfloor ns \rfloor)^2.
\end{align*}
Next Theorem 1 of \cite{Mor} together with the moment inequality stated above implies
\begin{align}
E \left ( \max_{ s \leq t \leq s + \delta} \left \lvert \sum_{i= \lfloor ns \rfloor +1}^{ \lfloor nt \rfloor} X_i \right \rvert \right)^4 \leq C (\lfloor n(\delta + s) \rfloor - \lfloor ns \rfloor)^2. \label{MorizcglGrIntbar}
\end{align}
Now we will show uniform integrability of (\ref{GlgrIntbar}). Using first H\"{o}lder- and Markov inequality and then (\ref{MorizcglGrIntbar}) one obtains
\begin{align*}
& E \left ( \max_{ s \leq t \leq s + \delta} \frac{1}{\delta} (W_n(t) - W_n(s))^2 1_{ \{ \max \frac{1}{\delta} (W_n(t) - W_n(s))^2 \geq K \} } \right ) \\
\leq & \ \frac{1}{K} \frac{1}{\delta^2} E \left ( \max_{ s \leq t \leq s + \delta} \lvert W_n(t) - W_n(s) \rvert \right)^4 \\
\leq & \ \frac{1}{K} \frac{1}{n^2 \delta^2} E \left ( \max_{ s \leq t \leq s + \delta} \left \lvert \sum_{i= \lfloor ns \rfloor +1}^{ \lfloor nt \rfloor} X_i \right \rvert \right)^4 \\
\leq & \ C \frac{1}{K} \frac{(\lfloor n(\delta + s) \rfloor - \lfloor ns \rfloor)^2}{n^2 \delta^2}.
\end{align*}
Because the last term tends to $0$ as $K \to \infty$, (\ref{GlgrIntbar}) is uniformly integrable and the partial sum process converges in $D[0,1]$ towards a Brownian Motion $W$ with
\begin{align*}
W(1) =_{\mathcal D} N(0,\sigma^2).
\end{align*}
Now consider an arbitrary separable Hilbert space $H$. For fixed $h \in H \setminus \{0 \}$, the sequence $( \langle X_i, h \rangle )_{ i \in \mathds N }$ is a sequence of real valued random variables. The mapping $x \mapsto \langle x, h \rangle$ is Lipschitz-continuous with constant $\lVert h \rVert$ and therefore by Lemma 2.2 of \cite{DeShWe}, $( \langle X_i, h \rangle )_{ i \in \mathds N }$ is $L_1$-near epoch dependent on an absolute regular process with approximation constants $(\lVert h \rVert a_m)_{m \in \mathds N}$ and has finite $(4 +\delta)$-moments, because
\begin{align*}
E \lvert \langle X_1, h \rangle \rvert^{4 + \delta} \leq \lVert h \rVert^{4 +\delta} E \lVert X_1 \rVert^{4 +\delta} < \infty.
\end{align*}
Thus we can apply the functional central limit theorem in $D[0,1]$ (proved in the lines above) and get
\begin{align}
\frac{1}{\sqrt n} \sum_{i=1}^{\lfloor nt \rfloor} \langle X_i,h \rangle \Rightarrow W_h(t), \label{FCLTInnerProd}
\end{align}
where $W_h$ is a Brownian motion with $E W_h(1)^2 = \sigma ^2 (h)$ and 
\begin{align*}
 \sigma ^2 (h)= \sum_{i= - \infty}^{\infty} E ( \langle X_0,h \rangle \langle X_i,h \rangle).
\end{align*}
Define the covariance operator $S  \colon H \to H$ by 
\begin{align*}
\langle S h_1, h_2 \rangle = \sum_{i=-\infty}^{\infty} E( \langle X_0,h_1 \rangle \langle X_i,h_2 \rangle).
\end{align*}
Then $\langle Sh,h \rangle = \sigma^2 (h)$ holds for all $h \in H \setminus \{0\}$. \\
Now we are able to verify condition (i) of Lemma A1. By the isometry and isomorphism between $H_k$ and $\mathds R^k$ it suffices to show for all $k \geq 1$
\begin{align}
\frac 1 {\sqrt{n}} \sum_{i=1}^{\lfloor nt \rfloor} Y_i^{(k)} \Rightarrow Y^{(k)}(t), \label{FCLTVektorRot}
\end{align}
where $Y_i^{(k)} = (\langle X_i, e_1 \rangle, \dots , \langle X_i, e_k \rangle)^t$ and $Y^{(k)}$ is Brownian motion in $\mathds R^k$, whose covariance matrix corresponds to $S^k=P_kSP_k$. By (\ref{FCLTInnerProd}) we obtain for all $k \geq 1$ and all $\lambda_1, \dots, \lambda_k$
\begin{align*}
\frac 1 {\sqrt{n}} \sum_{i=1}^{\lfloor nt \rfloor} \langle X_i, \sum_{j=1}^k \lambda_j e_j  \rangle \Rightarrow W_{\sum \lambda_j e_j}(t).
\end{align*}
But this implies (\ref{FCLTVektorRot}), because of the Cram\'{e}r-Wold device, the arguments used for verifying (\ref{fidiReell}) and the fact that univariate tightness in $D[0,1]$ implies tightness in $D_{\mathds R^k}[0,1]$.  Thus condition (i) of Lemma A1 is satisfied. For condition (ii) we need that $W^k \Rightarrow W$ as $k$ goes to $\infty$. But this holds, because
\begin{align}
\nonumber 
& W^k =_{\mathcal D} P_kW \\
\text{and} \ \ \ \ & \sup_{t \in [0,1]} \lVert P_k W - W \rVert \rightarrow 0 \ \ \text{a.s.}, \label{FastSicherKonvergenzBrBewegungRot}
\end{align}
for $k \to \infty$. (\ref{FastSicherKonvergenzBrBewegungRot}) holds pointwise due to Parseval's identity. The uniform convergence follows from the almost sure continuity of $W$.\\
Thus it remains to prove condition (iii). Define the operator $A_k \colon H \to H$ by $A_k = I - P_k$, where $I$ is the identity operator on $H$, and note that the mapping $h \mapsto A_k(h)$ is Lipschitz-continuous with Lipschitz-constant $1$. Thus $(A_k(X_i))_{i \in \mathds N}$ is a 1-approximating functional with the same constants as $(X_i)_{i \in \mathds N}$. From Lemma 2 it follows
\begin{align}
E \lVert A_k(X_1) + \dots + A_k(X_n) \rVert ^4 \leq C n^2 \left( E \lVert A_k(X_1) \rVert ^{4 + \delta} \right)^{\frac{1}{1 + \delta}}. \label{4MomentUngleichungA_k}
\end{align}
Observe that
\begin{align*}
E \left( \sup_{t \in [0,1]} \lVert W_n(t) - P_kW_n(t)\rVert ^4 \right)
= \frac{1}{n^2} E \left( \max_{1 \leq m \leq n} \left \lVert \sum_{i=1}^m A_k(X_i) \right \rVert^4 \right)
\end{align*}
and note that the term on the right hand side is bounded by $C \left( E \lVert A_k(X_1) \rVert ^{4 + \delta} \right)^{\frac{1}{1 + \delta}}$, due to (\ref{4MomentUngleichungA_k}) and Lemma A3. The constant $C$ does not depend on $k$ so it suffices to show
\begin{align}
E  \lVert A_k(X_1) \rVert ^{4 + \delta} \xrightarrow{k \to \infty} 0. \label{KovergenzA_kOperator}
\end{align}
By Parsevals's identity and the orthonormality of the $e_i$ one obtains
\begin{align*}
\lVert A_k(X_1) \rVert^2 =  \lVert \sum_{i=k+1}^{\infty} \langle X_1, e_i \rangle e_i \rVert^2 = \sum_{i=k+1}^{\infty} \langle X_1, e_i \rangle^2  \xrightarrow{k \to \infty} 0 \ \ \text{a.s.}
\end{align*}
Further  $\lVert A_k(X_1) \rVert ^{4 + \delta} \leq \lVert X_1 \rVert ^{4 + \delta} < \infty$ almost surely and thus, by dominated convergence, (\ref{KovergenzA_kOperator}) holds. But this implies condition (iii) of Lemma A1 and therefore finishes the proof.
\end{proof}

\begin{proof}[Proof of Theorem \ref{SeqBootstrapHilbertraum}]
Assume $EX_1=0$ and define
\begin{align*} 
& S_{n,i}^{\ast} := \frac{1}{\sqrt p} \sum_{j=(i-1)p +1}^{ip} (X_j^{\ast} - E^{\ast} X_j^{\ast}) \\
\text{and} \ \ & R_{n,kp}^{\ast} (t) := \frac{1}{\sqrt{kp}} \sum_{j= \lfloor kt \rfloor p +1}^{\lfloor kpt \rfloor} (X_j^{\ast} - E^{\ast} X_j^{\ast} ).
\end{align*}
Consider the following decomposition of the process $W_{n,kp}$ into the partial sum process of the independent blocks and the remainder
\begin{align*}
W_{n,kp}^{\ast} (t) = \frac{1}{\sqrt k} \sum_{i=1}^{\lfloor kt \rfloor} S_{n,i}^{\ast} + R_{n,kp}^{\ast} (t).
\end{align*}
We start by proving that $R_{n,kp}^{\ast}$ is negligible, i.e. 
\begin{align}
R_{n,kp}^{\ast}(\cdot )  \xrightarrow{P^{\ast}} 0 \ \ \text{a.s.} \label{Negligible}
\end{align}
uniformly as $n \to \infty$. Note, that $R_{n,kp}^{\ast} (t)$ is the sum over the first $l$ variables of a randomly generated block, where $l = l(k,p,t) = \lfloor kpt \rfloor - \lfloor kt \rfloor p$. Thus, for fixed $t$ we have
\begin{align*}
\lVert R_{n,kp}^{\ast} (t) \rVert \leq \frac{1}{\sqrt{kp}} \max_{1 \leq j \leq k} \left \lVert \sum_{i=j(p-1)+1}^{j(p-1)+l} (X_i- E^{\ast} X_i^{\ast} ) \right \rVert.
\end{align*}
Taking a supremum over $t$, we get
\begin{align*}
\sup_{t \in [0,1]} \lVert R_{n,kp}^{\ast} (t) \rVert \leq & \frac{1}{\sqrt{kp}} \max_{1 \leq j \leq k} \max_{1 \leq l \leq p} \left \lVert \sum_{i=j(p-1)+1}^{j(p-1)+l} (X_i- E^{\ast} X_i^{\ast} ) \right \rVert \\
=: & \ Y_{n,kp}.
\end{align*}
We will show, that $Y_{n,kp}$ converges to $0$, almost surely. \\
For $n \in \{2^{l-1} +1, \cdots , 2^l \}$ observe that
\begin{align*}
Y_n \leq & \frac{2}{\sqrt{2^l}} \max_{ j \leq k(2^l)} \max_{1 \leq m \leq p(2^l)} \left \lVert \sum_{i=j(p-1)+1}^{j(p-1)+m} (X_i- E^{\ast} X_i^{\ast} ) \right \rVert \\
=: & \ Y_l'.
\end{align*}
Taking the sum instead of the maximum, we can begin to bound the fourth moments of $Y_l'$:
\begin{align}
\nonumber E \lvert Y_l'\rvert^4 = & \  \frac{16} {2^{2l}} E \left(  \max_{ j \leq k(2^l)} \max_{m \leq p(2^l)} \left \lVert \sum_{i=j(p-1)+1}^{j(p-1)+m} (X_i- E^{\ast} X_i^{\ast} ) \right \rVert \right) ^4 \\
\nonumber \leq & \ \frac{16}{2^{2l}} \sum_{j=1} ^{k(2^l)} E \left( \max_{m \leq p(2^l)} \left \lVert \sum_{i=j(p-1)+1}^{j(p-1)+m} (X_i- E^{\ast} X_i^{\ast} ) \right \rVert \right) ^4 \\
= \nonumber & \ \frac{16k(2^l)}{2^{2l}} E \left( \max_{m \leq p(2^l)} \left \lVert \sum_{i=1}^{m} (X_i- E^{\ast} X_i^{\ast} ) \right \rVert \right) ^4.
\end{align}
The last line holds since $(X_i)_{ i \in \mathds N}$ and $E^{\ast} X_i^{\ast}$ does not depend on the block in which $X_i^{\ast}$ is, but only on the position of $X_i^{\ast}$ in this block. We want to make use of Lemma A3. For $p = p(2^l)$ and $k = k(2^l)$ we obtain using the Minkowski inequality
\begin{align*}
E \left \lVert \sum_{i=1}^{p} (X_i- E^{\ast} X_i^{\ast} ) \right \rVert  ^4
= & E \left \lVert \sum_{i=1}^{p} X_i - \frac{1}{k} \sum_{i=1}^{kp} X_i \right \rVert ^4 \\
\leq & \left \{ \left( E \left \lVert \sum_{i=1}^{p} X_i \right \rVert^4 \right)^{1/4} + \left( E \left \lVert \frac{1}{k} \sum_{i=1}^{kp} X_i \right \rVert^4 \right)^{1/4} \right \}^4 \\
= & O (p^2).
\end{align*}
In the last line we have used Lemma A2 and the fact that the first summand dominates. \\
Next by virtue of Lemma A3 we obtain 
\begin{align*}
E \left( \max_{m \leq p(2^l)} \left \lVert \sum_{i=1}^{m} (X_i- E^{\ast} X_i^{\ast} ) \right \rVert \right) ^4 = O(p^2).
\end{align*}
 Thus $E \lvert Y_l'\rvert^4 = O (\frac{p(2^l)}{2^l}) = O ((2^{- \epsilon})^l)$, because of $p(n) = O(n^{1- \epsilon})$, see the definition of the block length. Now an application of the Markov inequality and the Borel-Cantelli Lemma implies that
 \begin{align*} 
   Y_l' \xrightarrow{l \to \infty} 0 \ \ \text{a.s.}
   \end{align*}
Now $Y_n \leq Y_l'$ for  $n \in \{2^{l-1}, \cdots , 2^l \}$ and thus $Y_n$ converges almost surely to $0$ as $n$ tends to infinity. Finally this leads to
\begin{align*}
E^{\ast} (\sup_{t \in [0,1]} \lVert R_n^{\ast}(t) \rVert ) \leq E^{\ast} Y_n = Y_n \rightarrow 0 \ \ \text{a.s.}
\end{align*}
and thus we have proved (\ref{Negligible}). \\
To verify convergence of the bootstrap process in $D_H[0,1]$ it suffices to show that
\begin{align*}
V_{n,kp}^{\ast} (t) = \frac{1}{\sqrt k} \sum_{i=1}^{\lfloor kt \rfloor} S_{n,i}^{\ast}
\end{align*}
converges to the desired Gaussian process. 
\\\\
We first establish the finite dimensional convergence. For $0 \leq t_1 < \cdots < t_l \leq 1$ and $l \geq 1$ consider the increments
\begin{align*}
(V_{n,kp}^{\ast} (t_1), V_{n,kp}^{\ast} (t_2) - V_{n,kp}^{\ast} (t_1), \cdots , V_{n,kp}^{\ast} (t_l) - V_{n,kp}^{\ast} (t_{l-1})).
\end{align*}
Note that the random variables $S_{n,i}^{\ast}$ are independent, conditional on $(X_i)_{i \in \mathds{Z}}$, so it is enough to treat $V_{n,kp}^{\ast} (t_i) - V_{n,kp}^{\ast} (t_{i-1})$ for some $i \leq l$. By the consistency of the bootstrapped sample mean of $H$-valued data (see \cite{DeShWe}), there is a subset $A$ of the underlying probability space with $P(A)=1$, so that for all $ \omega \in A$ the central limit theorem holds:
\begin{align}
\frac{1}{\sqrt k} \sum_{i=1}^k S_{n,i}^{\ast} \Rightarrow N, \label{CLTHilbertGruen}
\end{align}
where $N$ is a Gaussian $H$-valued random variable with mean zero and covariance operator $S \colon H \to H$ defined by 
\begin{align*}
\langle S x, y \rangle = \sum_{i=-\infty}^{\infty} E [\langle X_0,x \rangle \langle X_i, y \rangle ], \ \ \text{for } x,y \in H.
\end{align*}
For $\omega \in A$ and arbitrary $t_i > t_{i-1}$ it follows by (\ref{CLTHilbertGruen}) that
\begin{align*}
V_{n,kp}^{\ast} (t_i) - V_{n,kp}^{\ast} (t_{i-1}) = & \ \frac{1}{\sqrt k} \sum_{i=\lfloor kt_{i-1} \rfloor +1}^{\lfloor kt_i \rfloor} S_{n,i}^{\ast} \\
= & \ \frac{ \sqrt {\lfloor kt_i \rfloor - \lfloor kt_{i-1} \rfloor}}{\sqrt k} \frac{1}{\lfloor kt_i \rfloor - \lfloor kt_{i-1} \rfloor}\sum_{i=\lfloor kt_{i-1} \rfloor +1}^{\lfloor kt_i \rfloor} S_{n,i}^{\ast} \\
\Rightarrow & \ \sqrt{t_i - t_{i-1}} N,
\end{align*}
where the distribution of $N$ is described previously. Thus the one dimensional distributions converge almost surely. But because of the conditional independence this implies the finite dimensional convergence.
\\\\
By Theorem A1, tightness will follow if we can show that
\begin{align}
\lim_{\delta \to 0} \limsup_{n \to \infty} \frac{1}{\delta} P^{\ast} \left ( \sup_{0 \leq t \leq \delta} \lVert V_{n,kp}^{\ast}(t) \rVert > \epsilon \right ) = 0 \ \ \text{a.s.} \label{Straffheit}
\end{align}
for all $\epsilon > 0$. \\
Using first Chebychev's inequality and then Rosenthal's inequality (see \cite{Ros} and \cite{LeTa} for validity in Hilbert spaces) we obtain
\begin{align*}
& \frac 1 {\delta} P^{\ast} \left ( \sup_{0 \leq t \leq \delta} \frac{1}{\sqrt k} \left \lVert \sum_{i=1}^{\lfloor kt \rfloor} S_{n,i}^{\ast} \right \rVert > \epsilon \right) \\
\leq & \ \frac{1}{\delta} \frac{1}{k^2 \epsilon ^4} E^{\ast} \left( \max_{1 \leq j \leq \lfloor k \delta \rfloor} \left \lVert \sum_{i=1}^j S_{n,i}^{\ast} \right \rVert^4 \right) \\
\leq & \ \frac{1}{\delta} \frac{1}{k^2 \epsilon ^4} C \left \{ \lfloor k \delta \rfloor E^{\ast} \lVert S_{n,1}^{\ast} \rVert^4 + \left ( \lfloor k \delta \rfloor E^{\ast} \lVert S_{n,1}^{\ast} \rVert^2 \right )^2 \right \} \\
\leq & \ C \frac{1}{\delta} \frac{k \delta}{k^2 \epsilon ^4} E^{\ast} \lVert S_{n,1}^{\ast} \rVert^4 + C \frac{1}{\delta} \frac{k^2 \delta^2}{k^2 \epsilon ^4} (E^{\ast} \lVert S_{n,1}^{\ast} \rVert^2 )^2 \\
= & \ I_n + II_n,
\end{align*}
where $I_n$ and $II_n$ are the respective summands. By the construction of the bootstrap sample and the the Minkowski inequality we get
\begin{align*}
I_n = & \frac{1}{k \epsilon ^4} \frac{1}{k} \sum_{i=1}^k \left( \frac{1}{\sqrt p} \left \lVert \sum_{j \in B_i} (X_j - \bar{X}_{n,kp})\right \rVert \right)^4 \\
= & \ C \frac 1{\epsilon^4} \frac 1 {k^2} \sum_{i=1}^k \left ( \frac1 {\sqrt{p}} \lVert \sum_{j \in B_i} X_j \rVert \right )^4 + C \frac 1 {\epsilon^4} \frac {p^2} k \lVert \bar X_{n,kp} \rVert^4 \\
& \ = \tilde I_{n,1} + \tilde I_{n,2}.
\end{align*}
By a strong Law of Large numbers (see Lemma 2.7 in \cite{DeShWe} ) we have
\begin{align*}
\frac {p^{1/2}} {k^{1/4}} \bar X_{n,kp} = \frac 1 {(kp)^{1/2} k^{1/2+1/4}} \sum_{i=1}^{kp} X_i \rightarrow 0 \ \ \text{a.s.},
\end{align*}
as $n \to \infty$. Hence $\tilde I_{n,2}$ converges almost surely to $0$. Regarding $\tilde I_{n,1}$, note that for $n \in \{2^{l-1}, \dots, 2^l\}$
\begin{align*}
 \tilde I_{n,1} \leq 16 C \frac 1 {\epsilon^4} \frac 1 {k(2^l)^2} \sum_{i=1}^{k(2^l)} \left ( \frac 1 {\sqrt{p(2^l)}} \lVert \sum_{j \in B_i} X_j \rVert \right )^4 :=  I'_{l,1}.
 \end{align*}
 We get by a fourth moment bound (see Lemma A2) 
 \begin{align*}
 E (I'_{l,1}) = O \left (1/k(2^l) \right ) = O \left (2^{- l \epsilon} \right ),
 \end{align*}
 because $k = \lfloor n/p(n) \rfloor$ and $p(n) = O(n^{1-\epsilon})$. Hence, by Markov's inequality and the Borel-Cantelli Lemma $I'_{l,1} \rightarrow 0$ almost surely for $l\to \infty$. Consequently $\tilde I_{n,1} \rightarrow 0$ almost surely for $n \to \infty$ and thus $I_1 \rightarrow 0$. \\
In \cite{DeShWe} it is shown that $E^{\ast} \lVert S_{n,i}^{\ast} \rVert^2 $ converges almost surely to $E \lVert N \rVert^2 $, where $N$ is Gaussian with the covariance operator defined above. Therefore $E \lVert N \rVert^2 $ is almost surely bounded and we obtain
\begin{align*}
II_n = \frac{\delta}{\epsilon^4} (E^{\ast} \lVert S_{n,1}^{\ast} \rVert^2 )^2 \xrightarrow{n \to \infty} \frac{\delta}{\epsilon^4} (E\lVert N \rVert^2)^2 \xrightarrow{\delta \to \infty} 0 \ \ \text{a.s.} 
\end{align*}
which implies (\ref{Straffheit}) and therefore finishes the proof.
\end{proof}

\begin{proof} [Proof of Corollary \ref{LimitTheoremCusumHilbertraumAlternative}]
Part (i) can be obtained by arguments similar to the case of real-valued random variables, see Theorem 2.1 in \cite{DeRoTa2}. 
\\\\
To verify part (ii) define random variables $U_1, \dots , U_k$, where $U_i$ is the number of the $i$th drawn block. Clearly the $U_i$ are all independent and uniformly distributed on $\{1, \dots , k\}$. \\
Note that the random variables in the blocks $B_1, \dots, B_{\lfloor k\tau \rfloor}$ are of the form $X_i$ and the variables of the blocks $B_{\lfloor k\tau \rfloor +2}, \dots, B_k$ are of the form $X_i +\Delta_n$. The change point occurs in the block $B_{\lfloor k\tau \rfloor+1}$, so this block contains shifted and non-shifted variables. \\
This subdivision in different types of blocks leads to the following decomposition of the process
\begin{align*}
\nonumber  \frac{1}{\sqrt {kp}} \left ( \sum_{i=1}^{\lfloor kpt \rfloor} Y_{n,i}^{\ast} - \frac{\lfloor kpt \rfloor}{kp} \sum_{i=1}^{kp} Y_{n,i}^{\ast}\right) 
 =  & \ \frac{1}{\sqrt{kp}} \left( \sum_{i=1}^{\lfloor kpt \rfloor} X_i^{\ast} - \frac{\lfloor kpt \rfloor}{kp} \sum_{i=1}^{kp} X_i^{\ast} \right) \\
 +& \ \sqrt{kp} \Delta_n R_{n,k,p}(t),
 \end{align*}
 where
 \begin{align}
 R_{n,k,p}(t) = & \  \frac{1}{kp} p  \sum_{i=1}^{\lfloor kt \rfloor} 1_{ \{ U_i > \lfloor k\tau\rfloor +1 \}} \label{BootsLocAlt1}\\
- & \  \frac{1}{kp} p  \frac{\lfloor kpt \rfloor}{kp}\sum_{i=1}^k 1_{ \{ U_i > \lfloor k\tau\rfloor +1 \}} \label{BootsLocAlt2} \\
+ & \  \frac{1}{kp} ((\lfloor kt \rfloor)+1)p - \lfloor n \tau \rfloor)  \sum_{i=1}^{\lfloor kt \rfloor} 1_{ \{ U_i = \lfloor k\tau\rfloor +1 \}} \label{BootsLocAlt3} \\
-& \ \frac{1}{kp} ((\lfloor kt \rfloor)+1)p - \lfloor n \tau \rfloor)  \frac{\lfloor kpt \rfloor}{kp} \sum_{i=1}^k 1_{ \{ U_i = \lfloor k\tau\rfloor +1 \}} \label{BootsLocAlt4} \\
+ & \ 1_{ \{ U_{ \lfloor kt\rfloor +1} > \lfloor k\tau\rfloor +1 \}} \frac{1}{kp} (\lfloor kpt \rfloor - \lfloor kt \rfloor p) \label{BootsLocAlt5}\\
+ & \ 1_{ \{ U_{ \lfloor kt\rfloor +1} = \lfloor k\tau\rfloor +1 \}} \frac{1}{kp} \max\{(\lfloor kpt \rfloor - \lfloor n \tau \rfloor p),0\} \label{BootsLocAlt6}.
\end{align}
By part (ii) of Corollary 1 and $\sqrt{kp} \Delta_n \rightarrow \Delta$ it remains to show that
\begin{align*}
P^{\ast} \left ( \sup_{t \in [0,1]} \lvert R_{n,k,p}(t) \rvert > \epsilon \right ) \rightarrow 0 \ \ \forall \epsilon > 0, \ \ \text{a.s.}
\end{align*}
as $n \to \infty$. But this holds because $R_{n,k,p}$ is independent of the $X_i$ and: 
(\ref{BootsLocAlt1}) $+$ (\ref{BootsLocAlt2}) and (\ref{BootsLocAlt3}) $+$ (\ref{BootsLocAlt4}) are each $o_P(1)$. To see this observe
\begin{align*}
\frac{1}{k} \sum_{i=1}^{\lfloor kt \rfloor} 1_{ \{ U_i > \lfloor k\tau\rfloor +1 \}} \xrightarrow{P} t(1- \tau),
\end{align*}
uniformly in $t$. \\
The quantity in (\ref{BootsLocAlt5}) is $o_P(1)$ because $(\lfloor kpt \rfloor - \lfloor kt \rfloor p)/(kp) \rightarrow 0$. Finally (\ref{BootsLocAlt6}) is $o_P(1)$ because $P(U_{ \lfloor kt\rfloor +1} = \lfloor k\tau\rfloor +1) =k^{-1}$.
 \end{proof}

\end{document}